\documentclass[twoside, 11pt, leqno]{amsart}
 \usepackage{amsmath, amssymb, amsfonts, amsthm, amscd, graphicx} \usepackage{mathrsfs}
 \usepackage{leftidx}
  \usepackage{xcolor}
  \usepackage{verbatim}

\long\def\symbolfootnote[#1]#2{\begingroup%
\def\thefootnote{\fnsymbol{footnote}}\footnote[#1]{#2}\endgroup}

\theoremstyle{plain}
\numberwithin{equation}{section}

\newtheorem{theorem}{Theorem}[section]  
\newtheorem{lemma}[theorem]{Lemma}
\newtheorem{corollary}[theorem]{Corollary}
\newtheorem{proposition}[theorem]{Proposition}
\newtheorem{definition}[theorem]{Definition}
\newtheorem{remark}[theorem]{Remark}

\newcommand{\fracsm}[2]{\begin{matrix}\frac{#1}{#2}\end{matrix}}

\newcommand{\beq}{\begin{equation}}
\newcommand{\eeq}{\end{equation}}

\newcommand{\Reals}{\mathbb{R}}
\newcommand{\Naturals}{\mathbb{N}}
\newcommand{\Integers}{\mathbb{Z}}

\newcommand{\Id}{\mathrm{Id}}

\DeclareMathOperator{\trace}{tr}

\title{Non-compact families of complete, properly immersed minimal surfaces with fixed topology via desingularization}
\date{\today}

\author{Stephen J. Kleene}
\address{Department of Mathematics, MIT, Cambridge, MA 02139.}
\email{skleene@math.mit.edu}

\author {Niels Martin M\o{}ller}
\address{Department of Mathematics, Fine Hall, Princeton University, NJ 08540.}
\email{moller@math.princeton.edu}

\begin{document}
\maketitle

\begin{abstract}
For fixed large genus, we construct families  of complete immersed minimal surfaces in $\mathbb{R}^3$ with four ends and dihedral symmetries. The families exist for all large genus and at an appropriate scale degenerate to the plane.
\end{abstract}

\section{Introduction}

The main result of this article is:
\begin{theorem} \label{MainTheorem}
There exists a family of complete minimal surfaces $\{ \Sigma (\theta, g)\}$ in Euclidean three-space with genus $g$ and four asymptotically catenoidal ends depending on parameters $\theta \in (0, \pi/2)$, and $g \in \mathbb{N}$. The surfaces exist for all $g$ sufficiently large and all $\theta$ sufficiently small, and depend continuously on $\theta$. Additionally, they have the following properties:
\begin{itemize}
\item[(1)] Away from the origin and the circle of unit radius about the origin in the plane $\{ x = 0\}$ they converge smoothly on compact subsets of $\mathbb{R}^3$ to the plane $\{x = 0\}$ with multiplicity four, as $\theta$ tends to $0$.

\item[(2)] Each $\Sigma(\theta, g)$ is invariant under rotations about the $z$-axis through angles $2 \pi/(g + 1)$ and the inversion through the plane $\{z = 0\}$  and reflections through the planes $\tan(y/x) =  \pi k / (g  + 1)$.

\item[(3)] Each $\Sigma (\theta, g)$ has four horizontal catenoidal ends, $E_1, \ldots, E_4$, which we order by height. The union of the catenoidal ends is close the configuration of two coaxial catenoids of scale $g^{-1}$ that intersect transversally along the circle of unit radius about the origin in the plane $\{ x= 0\}$, and the angle of their intersection is close to $2 \theta$.
\item[(4)] Let $\theta$ be fixed. Then the surfaces 
\begin{align}
\tilde{\Sigma} (\theta, g) : = g\left\{\Sigma (\theta, g)  - e_2 \right\}
\end{align}
 converge smoothly on compact sets to Scherk's singly-periodic minimal surface. 
 \end{itemize}
\end{theorem}

The modern theory of minimal surfaces with finite topology began with the discovery by C. Costa in his 1982 thesis, published in \cite{Co84}, of a genus one complete minimal surface with three ends that was apparently globally embedded. Hoffman-Meeks in \cite{HM85} discovered a family of analogous surfaces with three ends and positive genus and proved that these surfaces, including the Costa surface, were actually embedded. The Costa-Hoffman-Meeks surfaces were at the time the only complete embedded minimal surfaces with finite topology  other than the plane, the helicoid, and the catenoid, and the first with non-trivial topology.  If the logarithmic growth of each end of the Costa-Hoffman-Meeks surfaces is held fixed, then the surfaces converge up to rigid motions to the configuration of a catenoid intersecting a plane through the waist as the genus tends to infinity. If the surfaces are instead normalized by keeping the supremum of the norm of the second fundamental form fixed, the surfaces converge up to rigid motions to Scherk's singly-periodic minimal surface with infinite topology and four asymptotically flat ends. Motivated by this observation, Kapouleas in \cite{Ka97} constructed families of  minimal surfaces with many ends and high genus.

Hoffman-Meeks in \cite{HM90a} conjectured that the space $\mathcal{M}(g, r)$ of complete embedded minimal surfaces with genus $g$ and $r$ ends is empty when $g + 2 < r$, and Ros \cite{Ros06} conjectured for $r \geq 4$ that if $\mathcal{M}(g, r)$ is non-empty then it is non-compact.  The cases $r < 4$ are interesting and in many ways distinct from the general case. The space $\mathcal{M}(g, 3)$ has been classified entirely and has been shown to be non-compact (cf. \cite{HK}). In the case of two ends,  Schoen (\cite{Sc}) has shown that $\mathcal{M} (0, 2)$ contains only the catenoid, and that the spaces $\mathcal{M} (g, 2)$ are empty for $g > 0$. When $r = 1$, Meeks and  Rosenberg (\cite{MR1}) have shown that the helicoid is the unique genus zero surface with one end. Recently, Hoffman and White (\cite{HW}) constructed a genus one embedded minimal surface asymptotic to the helicoid, and Hoffman-White-Traizet in \cite{HTW1}--\cite{HTW2} constructed such surfaces for every positive genus.  

We emphasize three points about the embeddedness of the surfaces we construct:
\begin{enumerate}
\item A beautiful result of Ros (recorded as Theorem 3.3 in \cite{Ros06}) states that a complete embedded minimal surface with positive genus $g$ has at most $4g + 4$ symmetries in $O (3)$, with equality achieved only by the family of maximally symmetric 3-ended Costa-Hoffman-Meeks surfaces.  The surfaces which we construct achieve this equality, and are thus a-priori non-embedded. \\ \notag

\item This fact is, however, detectable also with our methods, and follows directly from Theorems \ref{MeanCurvatureStructure} and \ref{KernelProjection}. The non-embeddedness of the surfaces is a consequence of perturbations in the logarithmic growths of the ends of the initial configuration. A minor modification of our proof yields the perturbation term as a function of $g$ which determines up to first order the radius of the maximal  ball around the origin in which the surfaces are embedded. \\

\item A failure of the surfaces to be embedded is in this sense, and a modification of our construction is expected to produce families of \emph{embedded} minimal surfaces which leave every compact set of the interior of $\mathcal{M}(4, g)$. This can be achieved essentially by doubling the genus of the surfaces relative to the symmetry group at the penalty of losing the up-down reflectional symmetry of the configuration. The loss of this symmetry does not fundamentally change the analysis, and we expect our methods to apply more or less directly. 

\end{enumerate}

\section{Acknowledgements}

The authors would like to thank Nicos  Kapouleas for several helpful conversations. S.J. Kleene was partially supported by NSF award DMS-1004646. N.M. M\o{}ller was partially supported by NSF award DMS-1311795.

\section {Outline}
The starting point for our construction are configurations of two coaxial catenoids of the same scale, parametrized by their intersection angle, which we call $\theta$. We then follow the basic technique of \cite{Ka97} and obtain complete immersed minimal surfaces by removing small tubular neighborhoods of the intersection circle and replacing it with controlled deformations of Scherk's saddle towers, and perturbing the resulting smooth surface to minimality. The family of Scherk towers is naturally parametrized by $\theta$ and we write the corresponding surface as   $\Sigma_\theta$.  $\Sigma_\theta$ is then asymptotic to four affine half-planes and is symmetric with respect to the reflections through the coordinate planes, and the translation by $2 \pi e_x$.

 Essentially, the perturbation is achieved by solving the linear problem 
\begin{align}
\mathcal{L}[S] u = H
\end{align}
where $\mathcal{L}[S]  = \Delta[S] + |A[S]|^2$ is the stability operator on the initial surface $S$.  The graph $S+ u \nu[S]$ has mean curvature which is smaller in an appropriate sense.  One necessary consequence of the perturbing process is that  the initial surfaces undergo small changes in their asymptotics. Modulo a reflection across a plane, the surfaces have two ends, which we here refer to as the top and the bottom. In order for the surfaces  to be embedded,  the logarithmic growth of the bottom end must be  less than or equal to that of the top end, since otherwise they will eventually intersect. Since our initial configuration consists of two coaxial catenoids with the same scale, the logarithmic growth of both ends of the initial configuration are the same. Thus  in order to  determine the embeddness or non-embeddedness of the surfaces directly from the construction, we have to be able to predict with a high degree of accuracy the change in asymptotics induced by the perturbing process. We do this  by carefully studying the mean curvature of the initial surfaces, which is concentrated near the intersection circle. Modulo a discrete rotational symmetry, a dilation and  a rigid motion of $\Reals^3$, the initial surfaces are small perturbations of the fundamental domains of the Scherk towers $\Sigma$. Denoting the perturbing vector field by $\xi$, we express the mean curvature in linear and and higher order parts:
\begin{align}
H_{\xi}  = \mathcal{L}_\xi + \mathcal{R}_\xi
\end{align}
By the ``linear part'' we mean the linear change in the mean curvature of $\Sigma$ due to the addition of the vector field $\xi$. Since $\Sigma$ is minimal, the tangential part of the field amounts to a reparametrization of the  underlying surface, so that we can express
\begin{align}
\mathcal{L}_\xi = \tau \mathcal{L}_\Sigma \xi^\perp
\end{align} 
where $\mathcal{L}_\Sigma$ is the stability operator on $\Sigma$ and $\xi^\perp$  denotes the normal component of the variation field $\xi^\perp : = \xi \cdot \nu$. Predicting to which side the bottom ends of the initial surface ``want'' to change their logarithmic growth is then equivalent to determining the ``kernel content'' of the mean curvature. Let $\mathcal{S}$ be a fundamental domain for $\Sigma$ and let $\phi$ be a function satisfying $\mathcal{L}[\mathcal{S}] \phi = 0$ Then locally, the kernel content of the linear part is given by
\begin{align}
\int_{\mathcal{S}}\mathcal{L}[\mathcal{S}] \xi^\perp \phi = \int_{\partial \mathcal{S}} \phi\nabla \xi^\perp \cdot \eta - \xi^\perp \nabla \phi \cdot \eta,
\end{align} 
where $\eta$ is the outward pointing boundary unit  co-normal. The boundary terms can then in theory be computed exactly. Up to a quadratic remainder, the kernel content of the mean curvature can then be shown to be non-zero.  

There are technical difficulties in working with small-angle Scherk surfaces $\Sigma_\theta$. As $\theta$ tends to zero, the geometry of the surfaces degenerates (see Figure \ref{ScherkCollapse}), the curvature concentrates along the lattice $2 \pi  \mathbb{Z}$ on the $x$-axis away from which they converge to a plane of multiplicity two.  Understanding exactly how this degeneration takes place is important for obtaining workable bounds for the error term. The deformation field $\xi$ is large compared to the background geometry of regions of $\Sigma$ with high curvature. However, we show that most of the perturbing field is tangential, and effects the mean curvature to  a higher order which can be controlled.  Without separating out normal and tangential effects, we would produce estimates for the mean curvature which would be unstably large. 

\begin{center}
\begin{figure}[htb]\label{ScherkCollapse}
\includegraphics[width=\textwidth]{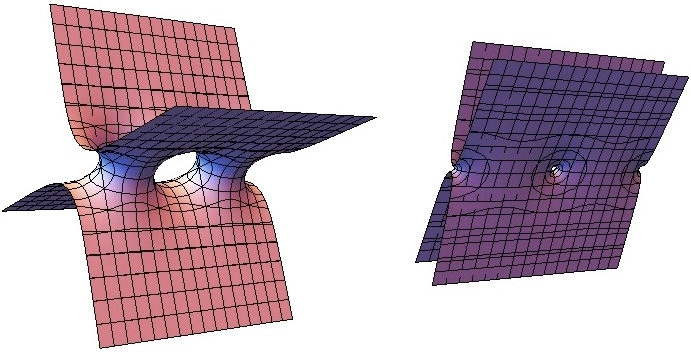}
\caption{Collapse of Scherk's singly-periodic minimal surface as the angle parameter $\theta\to 0$, or viewed from the right to the left showing the archetypical "doubling" of a flat plane as a complete embedded minimal surface in $\Reals^3$. With our normalization, the axis of periodicity is the $x$-axis. The $y$-axis points upwards and the plane to which the surfaces collapse is $\{ z = 0\}$.}
\end{figure}
\end{center}

In Section \ref{preliminaries} we introduce basic objects and notation which we use throughout. Additionally, we take some time to formalize various types of estimates which will arise repeatedly throughout the article. Specifically, we formalize the process of estimating the change, due to the addition of a vector field, of quantities defined on parametrizations that scale homogeneously, in terms of the scale of the parametrization and the  vector field. These include the mean curvature,  the unit normal, the components of the metric, the dual metric and the second fundamental form in coordinate charts,  and the coefficients of the Laplace operator  in coordinate charts.  In Section \ref{LaplacianOnCylinders}, we record several weighted invertiblity results for the Laplace operator on flat cylinders which we will use repeatedly throughout. In Section \ref{CatenoidalEnds}, we record a family of conformal parametrizations of catenoidal ends which we will use in our construction of the initial surfaces. In Section \ref{ScherkTowers}, we record properties and notation associated with the Scherk towers relevant to  our construction.  In particular, we record precisely how the geometry of the Scherk surfaces degenerates as the parameter $\theta$ tends to zero, which is needed for producing workable estimates for perturbations of geometric quantities on our approximate solutions. In Section \ref{StabilityOperatorOnCatenoid}, we record an invertibility result--Proposition \ref{CatenoidInverse}--for the stability operator on the catenoid in weighted H{\"o}lder spaces. Later, when we study the linear problem on the Scherk towers and record an analogous invertibilty statement in Proposition \ref{FlatInverse}, we will make use of Proposition \ref{CatenoidInverse}. In Sections \ref{BendingMapsSection}, \ref{MatchingSection} and \ref{TheInitialSurfaces}, we construct the initial surfaces  and record their basic geometric properties, including the 
criteria for smoothness and embeddedness, and their symmetry groups.  We break up the construction of the initial surfaces into these sections according to a natural set of independent technical considerations. The first--treated in Section \ref{BendingMapsSection}--has to do with the degeneration of the Scherk surfaces for small parameter values. In a fixed small ball about  the $z$-axis, they resemble large pieces of catenoids of scale approximately  equal to $\theta$. To avoid complicated geometric estimates on this region, we define bending maps which act as the identity in this region.  The second set of technical difficulties has to do with estimating the mean curvature of graphs over the catenoidal ends recorded in Section \ref{CatenoidalEnds}.   In Proposition \ref{InitialGraphProps} we record, among other things, precise conditions under which the initial surfaces are embedded and non-embedded.  In Section \ref{MeanCurvatureOfInitialSurfaces} we record a decomposition of the mean curvature of the initial surfaces and small normal graphs, into a ``linear'' and ``higher order'' part--Proposition \ref{MeanCurvatureStructure}. The ``linear'' part contains three principal parts: The linear change due to adding a graph, the linear change due to varying the controlling parameters on the initial surface, and the linear change due to ``bending'' the Scherk tower around a circle of large radius. The ``higher order'' part of the decomposition does not actually appear quadratically small in our estimates; nonetheless our estimates show that it is dominated by the terms constituting the linear part.  We also record Proposition \ref{KernelProjection}, which estimates the magnitude of the kernel content of the mean curvature. From this, the non-embeddedness of the surfaces could be deduced directly, without appealing to Ros's Theorem. In Section \ref{FlatScherkLinearProblem}, we record an invertibility statement for the stability operator on the Scherk towers. The construction of the initial surfaces is then concluded in Section \ref{FindingMinimalGraph} by a Schauder fixed point argument.

\section{Preliminaries} \label{preliminaries}

\subsection{Basic notation}
Throughout this article, $\mathbb{R}^3$ will denote Euclidean three-space, $X$ a point in $\mathbb{R}^3$ and $(x, y, z)$ the right-handed rectangular coordinates of that point, and $\{e_x, e_y, e_z \}$ the standard basis vectors.   We set\begin{align} \label{RadialField}
e_r (t) : = \sin (t) e_x + \cos (t) e_y,
\end{align}
The vectors $e_y[\beta]$ and $e_z[\beta]$ are given by
\begin{align} \label{RotatedFrame}
e_y[\beta] : = \cos (\beta) e_y + \sin (\beta) e_z, \quad e_z[\beta] : = \cos (\beta) e_z - \sin (\beta) e_y.
\end{align}
 
 We denote Euclidean two-space by $\mathbb{R}^2$, and take as coordinates $(x, s)$.

For real numbers $a$ and $b$, we set
\[
\psi[a, b] (s) = \psi_0 \left(\frac{s - a}{b - a}\right),
\]
where $\psi_0: \mathbb{R} \rightarrow [0, 1]$ is a fixed smooth, increasing function with $\psi_0\equiv 0$ on $(- \infty, 1/3)$ and $\psi_0\equiv 1$ on $(2/3, \infty)$. 


The half-spaces $H^{\pm} \subseteq \mathbb{R}^2$ are obtained by restricting the $s$-coordinate to either the non-negative or to the non-positive real values, respectively. We denote the flat cylinder 
\[
\Omega=\Reals^2/\langle x \mapsto x + 2 \pi\rangle.
\]

For a given subset $U$ of  $\mathbb{R}^2$ or of $\Omega_0$, we set
\begin{align}
U_{\leq c} : = U \cap \{ s \leq c \}, \quad U_{\geq c}  = U \cap \{ s \geq c\},
\end{align}
Similarly for a surface $S$ parametrized by $\phi: U\to\Reals^3$, we denote $S_{\leq c}=\phi(U_{\leq c})$ and so on.

\subsection{Geometric quantities on surfaces}

Typically, for a parametrized surface, we will denote the first and second fundamental forms by $g = (g_{ij})$ and $A = (A_{ij})$ respectively and the unit normal field and Christoffel symbols by $\nu$ and $\Gamma = (\Gamma_{ij}^k)$, respectively. We take the trace of the second fundamental form of a surface to be its mean curvature, and denote it by $H:=\trace A=g^{ij}A_{ij}$. For clarity, when corresponding to a surface $S$, these quantities and their components will typically  be paired with the symbol $[S]$, so for example $g[S]$ denotes the metric on the surface $S$ and $g[S]_ { ij} $ denotes its components in a coordinate neighborhood.

\subsection{Isometries and quotients}

\begin{definition} \label{IsometryDefs}
We let $\mathfrak{R}_x$, $\mathfrak{R}_y$,  $\mathfrak{R}_z$ denote the reflections through the coordinate planes $\{ x = 0\}$, $\{  y = 0\}$, and $\{ z = 0\}$, respectively. We let $\mathfrak{T}_t$ denote the translation by $ t e_x$, and we let $\mathfrak{T}^*_t$ denote the rotation
\begin{align} \notag
\mathfrak{T}^*_t (x, y, z) =  (\tau^{-1} + y)e_r (\tau t) - \tau^{-1} e_y + z e_z.
\end{align}
\end{definition}

\begin{definition} \label{IsometryGroupDefs}
We let $\mathfrak{G}$ denote the group of isometries generated by $\mathfrak{R}_x$, $\mathfrak{R}_z$ and $\mathfrak{T}_{2 \pi}$, and we let $\mathfrak{G}^*$ denote the group of isometries generated by $\mathfrak{R}_x$, $\mathfrak{R}_z$ and $\mathfrak{T}^*_{2 \pi}$.

\end{definition}

\begin{definition}
We let $\mathbb{E}$ denote the quotient of $\mathbb{R}^3$ by $\mathfrak{G}$, and we let $\mathbb{E}^*$ denote the quotient of $\mathbb{R}^3$ by $\mathfrak{G}^*$.
\end{definition}

Let $\mathfrak{G}$ be the isometry subgroup of $\Sigma$   generated by and $\mathfrak{T}$. We  denote by $\mathcal{S}$ the  quotient of $\Sigma$ under $\mathfrak{G}$  in the space $\mathbb{E} : = \mathbb{R}^3/\mathfrak{G}$.

\subsection{Weighted norms and H\"older spaces}

\begin{definition}
Given a function $u: D \subseteq \mathbb{R}^n \rightarrow \mathbb{R}$ we let $\| u : {C^{k, \alpha}( D)} \|$ denote the $k^{th}$ H\"older norm with exponent $\alpha\in (0,1)$, on the domain $D$. That is, we set
\begin{align}
\| u : {C^{k, \alpha}( D)} \| = \| u \|_{C^k (D)} + \sup_{\beta = k} [D^{\beta}u]_{C^{0, \alpha} (D)},
\end{align}
where
\begin{align}
[u]_{C^{0, \alpha}(D)} = \sup_{x, y \in D , x \neq y}\frac{|u (x) - u(y) |}{|x - y|^\alpha}.
\end{align}
\end{definition}

\begin{definition}
Given a function $u \in C^{j, \alpha} (D)$, the   $(j, \alpha)$ localized H\"o{}lder norm is given by
\begin{align}
\| u\|_{j, \alpha} (p) : = \| u : C^{j, \alpha} (D \cap B_1 (p))\|.
\end{align}
We let $C^{j, \alpha}_{loc} (D)$ denote the space of functions for which the  $(j, \alpha)$ localized H\"o{}lder norms are point-wise finite. 
\end{definition}

\begin{definition}
Given  a positive function $f: D \rightarrow \Reals$,  we let  $C^{j, \alpha} (D, f)$ be the space of functions for which the weighted norm $\| -: C^{j, \alpha} (D, f)\|$  is finite, where we take
\begin{align}
\| u : C^{j, \alpha} (D, f) \| : = \sup_{p \in D} \frac{1}{f(p)} \| u\|_{j, \alpha} (p).
\end{align}
\end{definition}

\begin{definition} \label{BanachIntersectionSpaces}
Let $\mathcal{X}$ and $\mathcal{Y}$ be two normed spaces with norms $\| - : \mathcal{X}\|$ and $\| - : \mathcal{Y}\|$, respectively. Then $\mathcal{X} \cap \mathcal{Y}$ is naturally a normed space with norm $\| -:  \mathcal{X} \cap \mathcal{Y}\|$ given by
\begin{align}\notag
\| f : \mathcal{X} \cap \mathcal{Y} \| = \| f: \mathcal{X}\| + \| f : \mathcal{Y}\|. 
\end{align}
\end{definition}

Frequently, we will want to measure functions and tensors that appear on various surfaces $S$. In all cases, we will identify and fix an atlas $\{ \phi_i:  D_i \rightarrow \Omega_i\}_{i = 1}^n$ of coordinate charts on $S$. When this is the case, we will set
\[
\| u: C^{k, \alpha}(\Omega_i, w ) \| : = \| u \circ \phi_i: C^{k, \alpha} (D_i, w \circ \phi_i )\|, 
\]
and
\[
\|u: C^{k, \alpha} (S, w ) \| : = \Sigma_{i = 1}^n \| u: C^{k, \alpha}(\Omega_i, w) \|.
\]
where $u: S \rightarrow \mathbb{R}$ is given and $w: S \rightarrow \mathbb{R}$ is a fixed weight function. 
Once an atlas for a surface has been fixed, we can measure tensors by measuring the norms of their components in the coordinate charts. Thus, when a tensor $T$ on $S$ is given, we set
\[
\| T: C^{k, \alpha} (\Sigma, w) \| : = \sum_{i = 1}^n \sum_{a,b} \| T_a^b: C^{k, \alpha}(\Omega_i, w ) \|,
\]
where $T_a^b$ ranges over all components of $T$ in the coordinate chart $\phi_i: D_i \rightarrow \Omega_i$. 

\subsection{Estimates of homogeneous quantities}
In many places in this article, we will have to produce estimates for weighted $C^k$ and $C^{k, \alpha}$ norms of both tensorial as well as non-tensorial quantities, such as the first fundamental form $g_{ij}$, second fundamental form $A_{ij}$, the unit normal $\nu$, and the Christoffel symbols $\Gamma^k_{ij}$ on various families of surfaces. In order to streamline our computations, we will make use of a few general properties of homogeneous functions.

\begin{definition} \label{JetSpace}
Let $\mathcal{J}$ be the Euclidean space
\[
\mathcal{J} : = \mathcal{J}^{(1)} \times \mathcal{J}^{(2)} =   (\Reals^{3})^2 \times (\Reals^{3})^4,
\]
which we think of as the formal jets up to second order of $C^2$-maps of open subsets of $\Reals^2$ to $\Reals^3$. We denote points of $\mathcal{J}$ by $ J = (J^{(1)}, J^{(2)} )$, where $J^{(1)}$ has two $\Reals^3$-elements and $J^{(2)}$ has four,
\begin{align}
J = \big(J^{(1)}_1, J^{(1)}_2\big), \qquad J^{(2)} = (J^{(2)}_{1 \, 1}, J^{(2)}_{2 \, 2}, J^{(2)}_{1 \, 2}, J^{(2)}_{2 \, 1}),
\end{align}
which we will soon think of as place-holders for the gradients and Hessians of an immersion, respectively.

We denote the natural Euclidean norms on these quantities by $|J_1^{(1)}|,|J_2^{(2)}|$,
\begin{align*}
&|J^{(1)}|=\sqrt{|J_1^{(1)}|^2+|J_2^{(2)}|^2},\\
&|J^{(2)}|=\sqrt{\sum_{i,j}|J_{ij}^{(2)}|^2}.
\end{align*}

\end{definition}

\begin{definition}
Let $\Phi:\mathscr{C}\subseteq\mathcal{J}\to\Reals$, where $\mathscr{C}$ is an open conical subset, with the property that for some $d\in\Integers$,
\begin{align}
\Phi(c J) = c^d \Phi(J), \quad c\in\Reals_+, \quad J\in \mathscr{C},
\end{align}
Then $\Phi$ is said to be a homogeneous function of degree $d$.
\end{definition}
Note that such a function which is homogeneous of degree $d$ and sufficiently smooth, has the property that for any multi-index $\beta$ (in the variables on which $\Phi$ depends), the function $\Psi=D^\beta\Phi$ is also homogeneous, of degree $d - |\beta|$.

\begin{definition} \label{PhiJets}
 Given an immersion $\phi : D \subseteq\Reals^2 \rightarrow \Reals^3$, we set 
 \begin{align}
 J^{(k)}[\phi] : =   \nabla^{k} \phi, \quad J [\phi] :  = (\nabla \phi, \nabla^{2} \phi),
 \end{align}
where $\nabla\phi$ and $\nabla^2 \phi$ are respectively the component-wise Euclidean gradient and Hessian of $\phi$, viewed as a map from the coordinate chart in $\Reals^2$ into $\Reals^3$.

A homogeneous quantity of degree $d$ on the surface $\phi(D)$ is then a function of the form $ \Phi_\phi : = \Phi (J[\phi])$ for some homogeneous function $\Phi$ on $\mathcal{J}$.
 \end{definition}
Examples of such quantities are many in surface geometry, and we record a few in the lemma below.

\begin{lemma} \label{degrees_of_homogeneity}
Let $g, g^{-1}, \nu, H$,  $\Gamma_{ij}^k, A, |A|$, $\Delta$ and $\mathcal{L}$ be the metric, the dual metric, the unit normal, the mean curvature, and the Christoffel symbols, the second fundamental form, the norm of the second fundamental form, the Laplace operator and the stability operator on an orientable surface, respectively. Then the coefficients of each (computed in a local chart), are homogeneous quantities of order $2$, $-2$, $0$, $-1$, $1$, $-1$, $-2$, $-2$ respectively.
\end{lemma}

We want to estimate the linear and higher order changes of homogeneous quantities along $\phi$ due  to the addition of small vector fields. To do this concisely,   we refer to a map $J: D \subseteq\Reals^2 \rightarrow \mathcal{J}$ as $\mathfrak{a}$-regular if the quantity
\beq \label{fraka_def}
\begin{split}
\mathfrak{a} (J) = \mathfrak{a} \big(J^{(1)}\big) : &= 2 \sqrt{\det\left[(J^{(1)})^T \, J^{(1)}\right]}/|J^{(1)}|^2\\
&=2\frac{\sqrt{|J^{(1)}_1|^2 |J^{(1)}_2|^2 -(J^{(1)}_1 \cdot J^{(1)}_2)^2}}{|J^{(1)}_1|^2 + |J^{(1)}_2|^2}.
\end{split}
\eeq
is everywhere non-zero, and otherwise we refer to it simply as a vector field. We initially note that:
\begin{lemma}
The quantity $\mathfrak{a}:\mathcal{J}\to\Reals$ is homogeneous of degree 0, and it holds that $0\leq\mathfrak{a}(J)\leq 1$. The first inequality is sharp when $J(\varphi)$ is a regular parametrization. Moreover, the second equality holds if and only if $|J^{(1)}_1| = |J^{(1)}_2|$ and $J^{(1)}_1 \cdot J^{(1)}_2 = 0$. In particular, $\mathfrak{a} (J(\phi)) = 1$ if and only if $\phi$ is a conformal immersion.
\end{lemma}
\begin{proof}
Since $\mathfrak{a}$ is clearly homogeneous degree of $0$, it suffices to consider the case that $|J^{(1)}_1|  = 1$ and $|J^{(1)}_1| : = r$. We can then write, if we let $\vartheta$ be the angle between $J^{(1)}_1$ and $J^{(1)}_2$,
\begin{align}
\big(\mathfrak{a}(J)\big)^2  = \big(\mathfrak{a} \big(J^{(1)}\big)\big)^2& = 4\frac{ |J^{(1)}_1|^2 |J^{(1)}_2|^2 -( J^{(1)}_1 \cdot J^{(1)}_2)^2}{(|J^{(1)}_1|^2 + |J^{(1)}_2|^2)^2}  \\ \notag
& = 4 \frac{\Reals^2 - \Reals^2\cos^2\vartheta}{(1 + \Reals^2)^2} = 4 (1 - \cos^2\vartheta) \frac{\Reals^2}{(1 + \Reals^2)^2}
\end{align}
For each $\vartheta$, the right hand side achieves a unique maximum at $r = 1$ with the value $(1 - \cos^2 \vartheta)$, which gives the claim.
\end{proof}

For the next definition, we identify $\mathcal{J}\simeq \Reals^{18}$, so that $\Phi: \Reals^{18}\to\Reals$, $E(p)\in\Reals^{18}$ for $p\in D$ and the multi-indices $\beta=(\beta_1,\ldots,\beta_k)$ and derivatives $D^{\beta}\Phi$ are w.r.t. these coordinates or components, e.g. also $(E(p))^\beta=(E(p))^{\beta_1}(E(p))^{\beta_2}\cdots(E(p))^{\beta_k}$.

 \begin{definition}
 Let $\Phi$ be a homogeneous quantity. Given an $\mathfrak{a}$-regular map $J:D\rightarrow \mathcal{J}$, and a vector field $E:D \rightarrow \mathcal{J}$, we let (for $k\in\mathbb{N})$
\begin{align} \notag
p\in D:\quad \left[R_{\Phi, E}^{(k)} (J)\right](p) : =  \sum_{|\beta|=k}(E(p))^\beta\int_0^1\frac{k}{\beta!}(1-\sigma)^{1-k}D^\beta\Phi_{\big| J(p)+ \sigma E(p)}d\sigma.
\end{align}
When $J$ and $E$ are of the form $J  = J [\phi]$ and $E  = J [V]$, for an immersion $\phi:D\to\Reals^3$ and a vector field $V:D\to\Reals^3$, we write:
 \begin{align} \notag
 R_{\Phi , V}^{(k)} (\phi) : = R_{\Phi , E }^{(k)} (J).
 \end{align}
 \end{definition}
 
 Note that $R_{\Phi , E}^{(k)} (J) $ is simply the Taylor remainder of order $k$, so that:
 \begin{proposition}
 We have, pointwise in $D$:
 \begin{align*}
 \Phi (J(p) + E(p)) = \Phi(J(p)) + \sum_{1\leq|\alpha|\leq k}\frac{D^{\alpha}\Phi(J(p))}{\alpha!}(E(p))^\alpha + R^{(k)}_{\Phi, E} (J)(p).
 \end{align*}
 \end{proposition}
 
We will assume throughout that any homogeneous function $\Phi$ considered is uniformly bounded in any $C^k$-norm on compact subsets of the set
\begin{align}
\mathcal{J}_0 : = \{J \in \mathcal{J} : \mathfrak{a} (J) > 0  \}\subseteq \Reals^{18}.
\end{align}

\begin{proposition} \label{still_regular}
Let $J^{(1)}$ and $E^{(1)}$ be points in $\mathcal{J}^{(1)}$ satisfying
\begin{align}
|E^{(1)}| \leq \varepsilon \mathfrak{a} \big(J^{(1)}\big) | J^{(1)}|.
\end{align}
Then for $\varepsilon>0$ sufficiently small , it holds that
\begin{align}
\left|\mathfrak{a} \big(J^{(1)} + E^{(1)}\big) - \mathfrak{a} \big(J^{(1)}\big) \right| \leq C |E^{(1)}|/ |J^{(1)}|,
\end{align}
where $C>0$ is a numerical constant (so independent of $J$ and $E$).
\end{proposition}
\begin{proof}
We have directly from the definition of $\mathfrak{a} \big(J^{(1)}\big)$ that
\begin{align}
\frac{\mathfrak{a} \big(J^{(1)}\big)}{C} \leq \frac{|J^{(1)}_1|}{|J^{(1)}_2|} \leq \frac{C}{\mathfrak{a} \big(J^{(1)}\big)},
\end{align}
for some $C>0$ (where e.g. $C=2$ works).

This then gives (e.g. $C=4$ works)
\begin{align} \label{SR1}
\Big(1 - \frac{C}{\mathfrak{a}^2} \Big)|J^{(1)}_2|^2 \leq |J^{(1)}|^2 \leq \Big(1 + \frac{C}{\mathfrak{a}^2} \Big)|J^{(1)}_2|^2, \\ \notag
\Big(1 - \frac{C}{\mathfrak{a}^2} \Big)|J^{(1)}_1|^2 \leq |J^{(1)}|^2 \leq \Big(1 + \frac{C}{\mathfrak{a}^2} \Big)|J^{(1)}_1|^2, 
\end{align}
where we set $\mathfrak{a}=\mathfrak{a} \big(J^{(1)}\big)$. Then
\begin{align} \label{SR2}
|E^{(1)}_1|^2 & \leq \varepsilon^2 \mathfrak{a}^2\Big(1  + \frac{C}{\mathfrak{a}^2}\Big) |J^{(1)}_1|^2 \leq \fracsm{1}{4} |J^{(1)}_1|^2, \\ \notag
|E^{(1)}_2|^2 & \leq \fracsm{1}{4} |J^{(1)}_2|^2,
\end{align}
by taking (if $C=4$), $0<\varepsilon<\fracsm{1}{4\sqrt{2}}$ and also assuming (since $|\mathfrak{a}| \leq 1$) $\varepsilon^2 \leq \frac{1}{8}$.

Set $J (\sigma) : = J + \sigma E$, for $\sigma \in [0, 1]$. Then from (\ref{SR1}) and (\ref{SR2}) we get 
\beq
\begin{split}
&\fracsm{2}{3}|J^{(1)}_2|  \leq |J^{(1)}_2 (\sigma)| \leq \fracsm{3}{2}|J^{(1)}_2|,\\
&\fracsm{2}{3}|J^{(1)}_1|  \leq |J^{(1)}_1 (\sigma)| \leq \fracsm{3}{2}|J^{(1)}_1|,
\end{split}
\eeq
so that $\fracsm{2}{3}|J^{(1)}|\leq |J^{(1)} (\sigma)|\leq\fracsm{3}{2}|J^{(1)}|$, meaning that these two quantities mutually control each other.

It is then straightforward to check that owing to homogeneity of $\mathfrak{a}$ and the above mutual control property, $|J^{(1)}| |D^\beta \mathfrak{a}_{|J^{(1)} (\sigma)}|$, $|\beta|=1$, is uniformly bounded for $\sigma \in [0, 1]$, so that 
\begin{align}
\left|\mathfrak{a} (J^{(1)} + E^{(1)}) - \mathfrak{a} \big(J^{(1)}\big) \right| = \left| R^{(1)}_{\mathfrak{a}, E^{(1)}}(J^{(1)})\right| \leq C |E^{(1)}|/|J^{(1)}|,
\end{align}
which gives the claim.
\end{proof}

\begin{proposition} \label{HQEstimates}
Let $\Phi:\mathscr{C}\subseteq\mathcal{J}_0\to\Reals$ be homogeneous of degree $d$. There is some $\varepsilon>0$ such that if we suppose that $J: D \rightarrow \mathcal{J}_0$ and $E: D \rightarrow \mathcal{J}$ satisfy
\begin{align} \notag
\|E: C^{j, \alpha}(D,\mathfrak{a} (J) |J^{(1)}|)\| \leq C(j,\alpha)\varepsilon , \quad \ell_{j, \alpha} (J) : =  \| J: C^{j, \alpha}(D, |J^{(1)}|)\|  < \infty,
\end{align}
for $0<\alpha<1$ and $j\in\Naturals$, then it holds that for any $k\in\Naturals$:
\begin{align} \notag
\left\| R_{\Phi , E}^{(k)} (J): C^{j, \alpha} (D, |J^{(1)}|^d) \right\| \leq C(k,\Phi,  \ell_{j, \alpha}, \mathfrak{a}(J) )\left\| E : C^{j, \alpha} (D, |J^{(1)}| ) \right\|^{k+1}.
\end{align}
\end{proposition}
\begin{proof}
Since $D^{\beta} \Phi $, $|\beta|=k$, is homogeneous of degree $d - k$, we can at each point $p\in D$ write

\[
R_{\Phi, E}^{(k)} (J) =  |J^{(1)}|^{d - k}\sum_{|\beta|=k}E^\beta\int_0^1\frac{k}{\beta!}(1-\sigma)^{1-k}D^\beta\Phi_{\big| J(\sigma)/|J^{(1)}|}d\sigma,
\]

where as before we have set $J(\sigma) : = J + \sigma E$. We then have
\begin{align} \label{HQ0}
\left|  D_p^\gamma R_{\Phi , E}^{(k)} (J) \right| & \leq C\left\| |J^{(1)}|\right\|_{j, \alpha}^{d - k} \sum_{|\beta|\leq k}\left\| D^{\beta} \Phi_{\big|J (\sigma)/| J^{(1)} |} \right\|_{j, \alpha} \left\|  E \right\|_{j, \alpha}^k,
\end{align}
where $D_p^\gamma$ is any combination of derivatives on the domain $D$ with $|\gamma|\leq j$ (and the same estimate holds for the H\"older ratio).

 By Proposition \ref{still_regular} and the assumptions of Proposition \ref{HQEstimates}, we get that $J (\sigma)/ |J^{(1)}|$ remains in a fixed compact subset of $\mathcal{J}_0$, and derivatives $D_p^\gamma[J (\sigma)/| J^{(1)}|]$, for $|\gamma|\leq j$, are controlled (via the assumptions involving $\varepsilon>0$ and $\ell_{j, \alpha}$), so that all in all we have bounds on the quantity
\begin{align} \label{HQ1}
\left\| D^{\beta} \Phi_{\big|J (\sigma)/| J^{(1)}|} \right\|_{j, \alpha}=\left\|(D^{\beta} \Phi)\circ \big(J (\sigma)/|J^{(1)}|\big)\right\|_{j,\alpha}\leq C (\Phi, \ell_{j, \alpha}, \mathfrak{a}(J),\varepsilon).
\end{align}
Additionally, we have that 
\begin{align}
p\in D:\quad \frac{\left\| |J^{(1)}|\right\|_{j, \alpha}^{d - k}}{|J^{(1)}|^{d - k} (p)} \leq C(1 +  \ell_{j, \alpha}^k)
\end{align}
Dividing both sides of (\ref{HQ1}) by $|J^{(1)}|^{d}$ then gives the claim.
\end{proof}

\subsection{Fixed point theorems}
In this section we record several standard fixed point theorems which we will use. The first, Proposition \ref{SchauderFixedPointTheorem}, is simply The Schauder Fixed Point Theorem, and whose proof can be found throughout the standard literature (c.f. GT). The second, Proposition \ref{IFTHandyVersion} is a basic corollary of the Contraction Mapping Theorem, which we state here for convenient application throughout the article. 

\begin{proposition} \label{SchauderFixedPointTheorem}
Every continuous  mapping from a convex compact subset $K$ of a Banach space to $K$ has a fixed point. 
\end{proposition}

\begin{proposition}[Approximate implicit function theorem with bounds] \label{IFTHandyVersion}
Let $\Phi(x, y): X \times Y \rightarrow \mathbb{R}^l$ be a class $C^{j, \alpha}_{loc}$ function for contractible  open subsets  $X \subset \mathbb{R}^l$  with $0 \in X$ and $Y \subset \Reals^{m}$ of Euclidean spaces. Assume that
\begin{enumerate}
\item $\| \Phi (x, y): C^{j, \alpha} (X \times Y) \| \leq \ell$.
\item $|\det \partial_x \Phi  (x, y)| >  B$.
\end{enumerate}
Then, given $B$ and $\ell$ as above, there is $A_0 = A_0 (B, \ell)$  so that: If  it holds that 
\begin{align}
\sup_{y \in Y}| \Phi (0, y) | : = A  \leq A_0, 
\end{align}
then there is a  map $y \mapsto x(y)$ so that:
\begin{align} \notag
\Phi (x(y), y) = 0,   \quad \| x (y) \|_{j, \alpha} \leq C A, \quad y \in Y.
\end{align}
\end{proposition}
\begin{proof}
We can write
\begin{align}
\Phi (x + h, y) = \Phi (x, y) + \Phi_x|_{(x, y)}(h) + R
\end{align}
where the assumptions then give 
\begin{align}
|R| \leq \ell |h|^2.
\end{align}
Thus, for $|x| \leq A C/\ell$, it holds that 
\begin{align}
|\Phi (x, y)| \leq 2 A.
\end{align}
Let $D(0, \rho) \subset \Reals^l$ be the ball of radius $\rho$ centered at $0$. With $\rho < A C/\ell $ we then define the map $\Psi (x, y) : D(0, \rho) \times Y \rightarrow \Reals^l$ to be given by:
\begin{align}
 \Psi (x, y) : = x -  \left( \Phi_x|_{(x, y)}(h) \right)^{-1}( \Phi (x, y)).
 \end{align}
We then have
 \begin{align} \notag
 \Psi (x + h) - \Psi (x) & = h - \left( \Phi|_{x \, (x + h, y)} \right)^{-1}( \Phi (x + h, y)) \\ \notag
 & = h + \left\{ \left(\Phi|_{x \, (x + h, y)} \right)^{-1} - \left( \Phi_x|_{(x, y)}(h) \right)^{-1}\right\}( \Phi (x + h, y)  \\ \notag
 & \quad -  \left( \Phi_x|_{(x, y)}(h) \right)^{-1}(  \Phi (x + h, y)  - \Phi (x , y)) \\ \notag
 & = h + O (\ell |h| A B^{-1})  - h + O (B^{-1}\ell|h^2|) \\ \notag
 & =  O (\ell |h| A B^{-1})  + O (B^{-1}\ell|h^2|) 
 \end{align}
 Thus, choosing $A$ small in terms of $B$ and $\ell$ gives that $\Psi$ acts as a contraction on $D(0, \rho)$ and hence has a unique fixed point, which we denote by $x(y)$. Uniqueness then implies continuous dependence on $y$. From continuity we get differentiability:
 \begin{align}
 h^{-1}\left(\Phi (x(y + h), y + h)  - \Phi (x(y), y) \right) &  =  h^{-1}\left(\Phi (x(y + h), y + h) -  \Phi (x(y), y + h) \right) \\ \notag
 & \quad  +  h^{-1} \left(\Phi (x(y), y + h) -  \Phi (x(y), y) \right) \\ \notag
 & = 0
 \end{align}
 Taking limits then gives
 \begin{align}
 x_y = - \left(\Phi_{x|(x, y)}\right)^{-1} \left( \Phi_{y|(x, y)}\right)
 \end{align}
 Higher order estimates for $x (y)$ then follow inductively.
\end{proof}

\section{Laplacian on cylinders} \label{LaplacianOnCylinders}

In this section we record several facts about the invertibility of the Laplace operator on a flat cylinder in various function spaces that we will use at various stages of this article.  The main results recorded in this section are Propositions \ref{LaplaceZeroMeridianInverse}, \ref{LaplaceInverseModDomain} and \ref{LaplacePerturbationInverseModDomain}. The uniting theme is a codification of several useful criteria which permit the Laplacian and nearby operators on the cylinder to admit an inverse in function spaces with decay. 

\begin{definition} \label{ZeroMeridianAverageSpaces}
Let $\mathring{C}^{k, \alpha}_{loc}(\Omega)$ denote the space of functions $E \in C^{k,\alpha}_{loc}(\Omega)$ satisfying the condition
\begin{align} \label{MeridianCondition}
\int_{- \pi}^\pi E(s, \theta) d \theta = 0
\end{align}
for all $s$. A function satisfying (\ref{MeridianCondition}) is said to have  zero average along meridians. Given a positive weight function $f$, we then denote
\begin{align*}
\mathring{C}^{k, \alpha} (\Omega, f) : = C^{k, \alpha} (\Omega, f)  \cap  \mathring{C}^{k, \alpha}_{loc}(\Omega).
\end{align*}
\end{definition}

\begin{proposition} \label{LaplaceZeroMeridianInverse}
Given $\rho \in (-1 , 1) \setminus \{ 0 \}$, $k \geq 0$, $\ell  \in (0, \infty]$ and  $\alpha \in (0, 1)$ there is a bounded linear map 
\begin{align}\notag
\mathcal{R}_1 [ -]: \mathring{C}_0^{0, \alpha} (\Omega_{\leq \ell}, \cosh^\rho (s)) \rightarrow \mathring{C}^{2, \alpha} (\Omega ,  \cosh^\rho (s)) \cap C^{k, \alpha}(\Omega_{ \geq  \ell + 1}, \cosh^\rho(\ell)/ \cosh (s)). 
\end{align}
such that:
\begin{enumerate}
\item $\Delta_{\Omega} \mathcal R_1[ E] = E.$

\item $\|\mathcal{R}_1 [E]  \| \leq C \| E \|$, where the norm on the target space is taken according to Definition \ref{BanachIntersectionSpaces}.
\end{enumerate}
\end{proposition}

\begin{proposition} \label{LaplaceInverseModDomain}
Given $\rho \in (0, 1)$, $k \geq 0$,  $\alpha \in (0, 1)$ and a bounded open set $D \subset \Omega$, there is a  well-defined bounded linear map
\begin{align}\notag
\mathcal{R}_2 [D,  -]: C^{0, \alpha} (\Omega ,  \cosh^{- \rho} (s)) \rightarrow C^{2, \alpha} (\Omega \setminus D ,  \cosh^{- \rho} (s)),
\end{align}
such that 

\begin{enumerate}
\item $\Delta_{\Omega} \mathcal R_2[D,  E] = E $.
\item $\|\mathcal{R}_2 [D,  E]  \| \leq C \| E \|$.
\item The map $\mathcal R_2[D,  E]$ depends continuously on $D$.
\end{enumerate}
\end{proposition}

\begin{proposition} \label{LaplacePerturbationInverseModDomain}
Let $\mathcal{L}$ be a second order linear operator and set
\begin{align} \notag
\| \mathcal{L} - \Delta_{\Omega}: C^{2, \alpha}_{loc}(\Omega) \rightarrow C^{0, \alpha}_{loc}(\Omega)\| : = \epsilon. 
\end{align}
Then, given $\rho \in (0, 1)$, $\alpha \in (0, 1)$ and a bounded open set $D \subset \Omega$, there is $\bar{\epsilon} > 0$ so that: For $\epsilon \leq \bar{\epsilon}$,  there is a  well-defined bounded linear map
\begin{align}\notag
\mathcal{R}_3 [\mathcal{L}, D,  -]: C_0^{0, \alpha} (\Omega \setminus D,  \cosh^{- \rho} (s)) \rightarrow C^{2, \alpha} (\Omega \setminus D ,   \cosh^{- \rho} (s)),
\end{align}
such that: 

\begin{enumerate}
\item $\Delta_{\Omega} \mathcal R_3[\mathcal{L}, D,  E] = E $.

\item $\|\mathcal{R}_3[\mathcal{L}, D,  E]  \| \leq C \| E \|$.

\item The map $\mathcal R_3[\mathcal{L}, D,  E]$ depends continuously on $D$ and $\mathcal{L}$.
\end{enumerate}
\end{proposition}

Propositions \ref{LaplaceZeroMeridianInverse} and  \ref{LaplaceInverseModDomain} can be constucted as corollaries to Lemma \ref{LaplaceCPTSPTInverse} below, while Proposition \ref{LaplacePerturbationInverseModDomain} follows from Proposition  \ref{LaplaceInverseModDomain} by standard perturbation techniques.  In the following, we let  $A_0 \subset \Omega$ be the annulus in the cylinder $\Omega$ given by:
\begin{align}\notag
A_0 : = \Omega \cap \{ |s| \leq 5/8 \}
\end{align}

\begin{lemma} \label{LaplaceCPTSPTInverse}

Given  a compact set $K$ containing $A_0$,  there is a bounded linear map
\begin{align}\notag
\mathring{\mathcal{R}}_0[- ]: \mathring{C}^{0, \alpha}_0 (A_0) \rightarrow C^{2, \alpha}(\Omega) \cap \mathring{C}^{k} (\Omega \setminus K, \cosh^{-1} (s)) 
\end{align}
such that
\begin{align}\notag
\Delta_\Omega \mathring{\mathcal{R}}_0[E] = E.
\end{align}
\end{lemma}
\begin{proof}
Lemma \ref{LaplaceCPTSPTInverse} can be established several ways. We choose the following approach. Let $\Omega_L$ be the domain 
\begin{align} \label{OmegaLDef}
\Omega_L : = \Omega\cap \{ |s| \leq L\}.
\end{align}
In other words $\Omega_L$ is just the  flat cylinder of length $2 L$ centered at the meridian $\{s = 0\}$. Standard elliptic theory gives the existence of  functions $u_L \in C^{2, \alpha}_{loc} (\Omega_L)$ satisfying:
\begin{align}\label {LCPTSPTI1}
\Delta u_L  & = E \\
 u_L (\theta, \pm L) & = 0
\end{align}
It is then direct to verify that the functions $u_L$ satisfy:
\begin{align} \label{LCPTSPTI2}
\int_{- \pi}^\pi u (s, \theta) d \theta = 0.
\end{align}
To see this, we integrate both sides of the first equality in  (\ref{LCPTSPTI1}) in $\theta$ to obtain
\begin{align} \notag
\left(\int_{- \pi}^\pi u (s, \theta) d \theta\right)_{s  s} = 0.
\end{align}
The boundary conditions  in (\ref{LCPTSPTI1}) then imply (\ref{LCPTSPTI2}). Elliptic estimates then give
 \begin{align} \label{LCPTSPT4}
 \|u_L: C^{2, \alpha}(K)\| \leq C(K) \| E : C^{0, \alpha} (K) \| \leq C(K) \| E : C^{0, \alpha} (A_0)\| .
\end{align}
Since $u_L$ is harmonic on $\Omega_L \setminus K$ with  Dirichlet condition on $u_L$ at $s = \pm L$  we then immediately get
\begin{align}
 \| u_L : C^k(\Omega_L \setminus K) \| \leq  C(K) \| E : C^{0, \alpha} (A_0)\| 
\end{align}
A subsequence $u_{L_j}$  then converges in $C^{2, \alpha'}(K)$, ($\alpha' < \alpha$) on compact subsets of $\Omega$ to a limiting function   $u_\infty$   satisfying
\begin{align}
\Delta_\Omega u_\infty = E, \quad \int_{- \pi}^\pi u_\infty(s, \theta) d\theta  = 0.
\end{align}
Standard regularity again gives that $u_\infty$ is in $C^{2, \alpha}(K)$. Since $u_\infty$  is uniformly bounded on $\Omega$ and has zero average along meridians,  the exponential decay   both the positive and negative $s$ directions follows directly from the absence of the zero mode in the Fourier expansion. We then set
\begin{align} \label{LCPTSPT5} 
\mathring{\mathcal{R}}_0  [E]=u_\infty (x, s).
\end{align}
  
\end{proof}

\begin{proof} [Proof of Proposition \ref{LaplaceZeroMeridianInverse}]
Fix $\mathring{E} \in \mathring{C}_0^{0, \alpha} (\Omega_{\leq \ell}, \cosh^\rho (s))$ and set
\[
 \beta : = \|\mathring{E}:  \mathring{C}_0^{0, \alpha} (\Omega_{\leq \ell}, \cosh^\rho (s)) \|
\]
  For each integer  $i\in\mathbb{Z}$, let $A_i $ be the annulus $A_i : = A_0 + i $.  Note that the set $\{ A_i\}$ is a locally finite covering of $\Omega$  such that $A_i \cap A_j = \emptyset$ if $|i - j| >1$. Let $\{ \psi_i\}$ be  a partition of unity subordinate to $\{ A_i \}$ such that $\psi_i (s + 1) = \psi_{i + 1} (s)$. Recall  that $\mathring{E}$ integrates to zero along  meridian circles :
\begin{align*}
\int_{- \pi}^\pi \mathring{E} (s, \theta) d\theta = 0.
\end{align*}
 With  $\mathring{E}_i (s, \theta): = \psi_i(s- i)\mathring{ E}(s- i, \theta)$, it is straightforward to check that  $\mathring{E}_i \in \mathring{C}_0^{0, \alpha}(A_0)$ with the estimate
\begin{align*}
\| \mathring{E}_i : \mathring{C}_0^{0, \alpha} (A_0)\| \leq  C \beta \cosh^{\rho} (i)
\end{align*}
We then set
\begin{align*}
\mathring{u}_i (s, \theta) : = \mathring{\mathcal{R}}_0 (E_i) (s + i, \theta).
\end{align*}
 From Lemma \ref{LaplaceCPTSPTInverse},
 \begin{align*} 
 \| \mathring{u}_i : C^{2, \alpha} (A_j)\| &  \leq C \beta \cosh^\rho( i)/  \cosh (j-i) \\ \notag
 & \leq C \beta e^{|\rho|| i|}/ e^{|j - i|}.
 \end{align*}

 \begin{align*}
 \sum_{i = -\infty}^\infty   \left\|\mathring{u}_i : C^{2, \alpha} (A_j)\right\| \leq \frac{C \beta}{1 - |\rho|} \cosh^\rho (j). 
 \end{align*}
Thus, being norm summable, the partial sums converge to a limiting function $\mathring{u}$ with zero average along meridians satisfying
\begin{align*}
\Delta_{\Omega} \mathring{u} = \mathring{E}, \quad \| \mathring{u}: C^{2, \alpha} (A_j)\| \leq \frac{C \beta}{1 - |\rho|} \cosh^\rho (j). 
\end{align*}
In other words $\mathring{u}$ satisfies the estimate
\begin{align*}
\| \mathring{u}: C^{2, \alpha} (\Omega, \cosh^{\rho} (s))\| \leq \frac{C}{1 - |\rho|} \| \mathring{E}: C^{0, \alpha} (\Omega, \cosh^{\rho} (s))\|
\end{align*}
Setting $\mathcal{R}_1[ \mathring{E}] : = \mathring{u}$ provides the result.
\end{proof}

\begin{proof}[Proof of Propositions \ref{LaplaceInverseModDomain} and \ref{LaplacePerturbationInverseModDomain}]
Since $E$ vanishes on the boundary of $D$ we may regard it as a $C^{0, \alpha}$ function on $\Omega$ after extending by $0$. Let $d$ denote the distance to the boundary $\partial D$ of $D$. We then set
\begin{align}
\psi (x, s) : = 
\left\{\begin{array} {c}
d^4, \quad (x, s) \in D \\
\\
0, \quad \mathrm{otherwise}
\end{array} \right.
\end{align}
Then $\psi$ is a $C^{2, \alpha}$ function that vanishes on the complement of $D$ and depends continuously on $D$. 
We then  set $f : = \psi s$ and $g : = \psi 1$.  From the Cauchy Schwartz inequality we then have
\begin{align} \notag
\det
\left(\begin{array}{cc}
\langle f, s\rangle & \langle f, 1 \rangle \\
\langle g, s \rangle & \langle g, 1 \rangle
\end{array}  \right) 
= \left( \int_\Omega \psi s^2\right) \left( \int_\Omega \psi 1\right) - \left( \int_\Omega \psi s \right)^2 > c_D > 0,
\end{align}
since $s$ and $1$ are linearly independent.  Thus, there are constants $a$ and $b$ satisfying
\begin{align}
|a|, |b| \leq C(D) \| E : L^2(\Omega) \| \leq C(D) \| E : C^{0, \alpha} (\Omega, \cosh^{-\rho}(s)) \|.
\end{align}
and depending continuously on $D$ so that $F : = E + a f + b g$ satisfies
\begin{align} \label{LIMD1}
\int_{\Omega} F & =  \int_{\Omega} F s = 0 \\ \notag
\| F : C^{0, \alpha} (\Omega, \cosh^{-\rho}(s)) & \| \leq C(D) \| E : C^{0, \alpha} (\Omega, \cosh^{-\rho}(s)) \|,
\end{align}
 Set $\bar{F}(s)  : = \frac{1}{2\pi}\int_{- \pi}^\pi F(x, s) dx $, $\mathring{F} : = F - \bar{F}$. Then $\mathring{F}$ belongs to the space $\mathring{C}^{0, \alpha}(\Omega, \cosh^{- \rho}(s))$. Recalling Proposition \ref{LaplaceZeroMeridianInverse}) in the case $\ell = \infty$,  we set
\begin{align} \notag
\mathring{v} : = \mathcal{R}_1[\mathring{F}] 
\end{align}
We also set
\begin{align}
\bar{v} (s): = \int_{s}^\infty \int_{s'}^\infty \bar{F}(s'') ds'' ds'.
\end{align}
It is clear that $\bar{v}$ decays like $\cosh^{- \rho}(s)$ in the positive $s$ direction. To establish decay in the negative $s$ direction, note that from the orthogonality relations in (\ref{LIMD1})  and $\theta$-independence of $1$ and $s$ we have
\begin{align}
\int_{-\infty}^\infty \bar{F}(s) ds  =  \int_{-\infty}^{\infty} \bar{F} (s) s ds = 0.
\end{align}
We can then write
\begin{align} 
\int_{s}^\infty \int_{s'}^\infty \bar{F}(s'') ds'' ds' & =  \int_s^\infty s' \bar{F}(s') ds'  -  s \int_s^\infty \bar{F}(s'') ds'' \\ \notag
& =  s \int_{-\infty}^s \bar{F}(s'') ds'' - \int_{-\infty}^s s' \bar{F}(s') ds'\\ \notag
& = \int_{-\infty}^s\int_{-\infty}^{s'} \bar{F}(s'') ds'' ds'.
\end{align}
Setting $\mathcal{R}_2[E] : = \bar{v} + \mathring{v}$ completes the proof of Proposition \ref{LaplaceInverseModDomain}. Proposition \ref{LaplacePerturbationInverseModDomain} is then a simple corollary using standard perturbation techniques.
\end{proof}

\section{Conformally parametrized catenoidal ends} \label{CatenoidalEnds}

In the following record a family of maps which conformally parametrize catenoids.  Note that the extremal  parameter $\beta = 0$ coincides with the standard conformal map from the cylinder to the flat plane and the extremal parameter $\beta = \pi/2$ agrees with the standard conformal parametrization of a scale one catenoid.
\begin{definition} \label{KappaDef}
Set $\varrho[\beta](s) = \cosh(s) + \cos (\beta) \sinh (s)$. Then the maps $\kappa[\beta] (x,s) : \Reals^2 \rightarrow \Reals^3$ are  given by
\begin{align} \notag
\kappa[\beta] (x, s) & : =  \varrho[\beta] (s) e_r (x) + \sin (\beta) s e_z[0].   \\ \notag
\end{align}
\end{definition}

In Proposition \ref{BasicKappaProps} below, we use will use the notion of \emph{logarithmic growth} of a catenoidal end $\kappa$, which is the unique multiple of  L of $\log(r)$ so that $\kappa - L \log (r)$ is bounded at infinity. 

\begin{proposition} \label{BasicKappaProps}
The maps $\kappa[\beta]$ have the following properties. 
\begin{enumerate}
\item  They are  each conformal minimal immersions with conformal factor $\varrho[\beta]$.
\item The image of each is a (scaled and translated) catenoid with axis of rotation equal to the $z$-axis. 
\item The half surface $\kappa[\beta] (\{ s \geq 0 \})$ is a catenoidal end with boundary equal to  the unit circle in the plane $\{ z = 0\}$  and logarithmic growth rate equal to $ \sin (\beta) $.
\end{enumerate} 
\end{proposition}

\begin{lemma} \label{TechnicalKappaProps}
Set $\varrho (s): = \varrho[\beta](s)$ for $\beta \in \Reals$. Then the following statements hold:
\begin{enumerate}

\item \label{abProps} Set $a (s)= \varrho'(s)/\varrho(s)$ and $b(s) = \sin(\beta)/\varrho(s)$. Then it holds that 
\begin{align}  \notag
a' = b^2, \quad b' = - ab, \quad a^2 + b^2 = 1.
\end{align}
\item \label{KappaFrame}The vectors
\[
e_1 (s) : = \kappa_x(x, s)/\varrho(s), \quad e_2(s) := \kappa_s(x, s)/\varrho(s), e_3 : = \nu[\mathcal{C}_0](x, s)
\]
 are a positively oriented orthonormal frame $\{e_i^{\kappa} \} = \{ e^{\kappa}_i [\beta]\}$ and we have explicitly
\begin{align} \notag
e_1^{\kappa}(x, s)  = e_r'(x), \quad e_2^{\kappa} (s)  = a (s) e_r (s) + b (s) e_z, \quad e_3^{\kappa} (x, s) = - b (s) e_r(x) + a(s) e_z.
\end{align}

\item \label{KappaCovarianceMatrix}Let $T_x^\kappa$ and $T_s^\kappa$ denote the  derivative matrices for the frame $\{e_i^\kappa \}$, so that $\partial_s e_i^{\kappa} = \left(T_{s}^\kappa \right)^j_i  e^{\kappa}_j$ and $\partial_x e_i^{\kappa} = \left(T_{x}^{\kappa} \right)_i^j  e^{\kappa}_j$. Then we have that 
\begin{align} \notag
T_x^\kappa = 
\left(\begin{array}{ccc}
0 & a & - b \\ 
-a & 0 & 0 \\
b & 0 & 0\\  
\end{array} \right), \quad
T_s^\kappa = 
\left(\begin{array}{ccc}
0 & 0 & 0 \\ 
0 & 0 & b \\
0 & - b& 0\\  
\end{array} \right)
\end{align}
\end{enumerate}
\end{lemma}

\begin{proof}
Claim (\ref{abProps}) is directly verified. To prove Claims (\ref{KappaFrame}) and (\ref{KappaCovarianceMatrix}), we write $\kappa = \varrho e_r  + \sin(\beta) s e_z$. The components of the gradient and hessian of $\kappa$ are then:
\begin{align} \notag
&\partial_x \kappa = \varrho e_r', \quad \partial_s\kappa = \varrho' e_r + \sin(\beta) e_z \\ \notag
& \partial^2_{x \, x}\kappa = - \varrho e_r = - \partial_{s \, s} \kappa, \quad \partial_{x \, s} \kappa = \varrho' e_r'.
\end{align}
Claim (\ref{KappaFrame}) then follows directly.  We have
\begin{align} \notag
& \partial_x e_1 = - e_r. \quad \partial_s e_1 = 0 \\ \notag 
& \partial_x e_2 = a e_r' \quad \partial_s e_2 = b^2 e_r - ab e_z \\ \notag
& \partial_x e_3 = -b e_r' \quad \partial_s e_3 = ab e_r + b^2 e_z,
\end{align}
from which claim (\ref{KappaCovarianceMatrix}) follows. 
\end{proof}
\subsection{Renormalized parametrizations} \label{RenormalizedMaps}

\begin{definition} \label{RenormKappaDefs}
Let $\tilde{\kappa}[\beta]$ be the renormalized map  given by
\begin{align} \notag
\tilde{\kappa}[\beta] (x ,s) : = \left(\kappa[\beta](\tau x, \tau s)  - \kappa[\beta](0, 0)\right)/ \tau.
\end{align}
\end{definition}

\begin{proposition} \label{RenormKappaBasicProps}
The maps $\tilde{\kappa}[\beta]$ have the following properties:
\begin{enumerate}
\item \label{RenormalizedConformalFactor} The maps  $\tilde{\kappa}[\beta]$ are conformal minimal immersions with conformal factor $\tilde{\varrho}[\beta] (s) : = \varrho[\beta] (\tau \, s)$.

\item The image of each is a catenoid with axis of rotation equal to the  line $\{ (0,-\tau^{-1}, t ) \}_{t \in \Reals}$.
\item The half surface $ \tilde{\kappa}[\beta] (\{ s \geq 0 \})$ is a catenoidal end with boundary equal to  the circle in the plane $\{ z = 0\}$ of radius $\tau^{-1}$ about the point $(0, - \tau^{-1}, 0)$  and logarithmic growth rate equal to $ \sin (\beta) /\tau $.

\end{enumerate}
\end{proposition}

\section{Scherk towers} \label{ScherkTowers}

\emph{Scherk towers} are a family of complete embedded minimal surfaces $\Sigma =\Sigma[\theta]$ given implicitly by 
\begin{align}\label{ScherkEq}
  \cos(x) = \cos^2(\theta) \cosh(y/ \cos (\theta)) - \sin^2(\theta)\cosh(z/ \sin (\theta))  
\end{align}
where $\theta$ belongs to the interval $(0, \pi/2)$.  In addition to minimality, the properties of these surfaces that are relevant to our construction are listed in plain language below:
\begin{enumerate}

\item The isometry group  of each surface contains the reflections $\mathfrak{R}_x$, $\mathfrak{R}_y$, $\mathfrak{R}_z$ through the coordinate planes and the translation $\mathfrak{T}_{2 \pi}$ by the vector $2 \pi e_x$.

\item \label{Asymptotics} Each surface is exponentially asymptotic to a collection of four half planes parallel to the $x$ axis.
\item \label{CloseToCatenoid}  In a fixed small tube about the $z$ axis, each surface is a perturbation of a large piece of a catenoid. 
\item \label{ConvergenceInTheta}Away from this tube about the $z$ axis, the surfaces are uniformly regular in $\theta$ and (up to vertical translation) converge smoothly to the plane $\{ z = 0 \}$ with multiplicity two.
\end{enumerate}
  Below we record quantitative versions of statements (\ref{Asymptotics}) and (\ref{CloseToCatenoid}) (\ref{ConvergenceInTheta}).  
  
  \subsection{Exponential convergence of Scherk towers to  four half planes} \label{QuantAsymptotics}

\begin{definition}
Let $H^+[\beta, h](x, s): H^+ \rightarrow \Reals^3 $ be the affine map given by:

\begin{align} \notag
H^+[\beta, h] (x, s) : = x e_x + \cos(\beta) e_y  + h e_z.
\end{align}
\end{definition}
Then the  surface $\mathcal{S}$ is asymptotic to $H^+[\beta, h]$ in the first quadrant (taken with respect to the $y-z$ axes) in the following sense:

\begin{proposition} \label{WingGeometry}
There is $\bar{\theta}  > 0 $ so that for each $\theta \in [0, \bar{\theta})$, there is a function $f_{\mathcal{W}}: H^+_{\geq 1} \rightarrow R$ such that:
\begin{enumerate}
\item \label{WingMap1Def}Set $h_{\mathcal{S}} : = \sin(\theta) \log (\cot^2(\theta))$ and let  $\mathcal{W}_1 (x, s): H^+_{\geq 1} \rightarrow \Reals^3$  be the map given by
\begin{align} \notag
\mathcal{W}_1 (x, s) : = H^+[\theta, h_{\mathcal{S}}] (x, s)  + f_{\mathcal{W}} e_z[\theta].
\end{align}
Then $\mathcal{W}_1$ maps $H^+$ into $\mathcal{S}$.
\item The intersection of  $\mathcal{S} \setminus \mathcal{W}_1$ with the first quadrant is contained in a fixed tubular neighborhood of the $x$-axis.
\item $f_{\mathcal{W}}$ satisfies the estimate
\begin{align} \notag
\| f_{\mathcal{W}} : C^{k, \alpha}(H^+_{\geq 1}, e^{-s})\| \leq C \sin (\theta).
\end{align}
\end{enumerate}
\end{proposition}

\begin{proof}
Set 
\begin{align} \notag
F(x, y, z) : = \cos(x) - \cos^2(\theta)\cosh(y/\cos(\theta)) + \sin^2(\theta) \cosh(z/\sin(\theta)),
\end{align}
so that $\mathcal{S}$ agrees with the zero set of $F$. We set
\begin{align} \notag
\Phi(f) : = F \circ (H^+[\theta, h] + f e_z [\theta]), 
\end{align}
We then have
\begin{align}
\Phi(0) & =  \cos (x) - \cos^2(\theta) \cosh(s) + \sin^2(\theta)\cosh (s + \log (\cot^2(\theta))) \\ \notag
& = \cos(x) - \cos^2(\theta) \cosh(s) + \frac{1}{2} \cos^2 (\theta) e^s  + \frac{1}{2} \sin^2 (\theta)\tan^2(\theta) e^{-s} \\ \notag
& = \cos(x) + O (e^{-s}) \\ \notag
\end{align}
Moreover, there is a constant $K$ so that 
\begin{align} \label{WG1}
\partial_f \Phi(f) > K \sin^{-1}(\theta) e^s 
\end{align}
as long as $f \leq 1$ (an arbitrary choice).  We then seek $f$ such that 
\begin{align}
0 = \Phi (f) = \Phi(0) + R^0_{\Phi, 0} (f).
\end{align}
where above $R^0_{\Phi, 0} (f)$ denotes the $0$ order Taylor remainder of $\Phi$ at $0$ evaluated at $f$. Note there is  a constant $L$ so that  $|\Phi(0) | \leq L$ for $s  \geq 1$. We then choose $\bar{\theta}$ sufficiently small so that
\begin{align} \notag
L/K \sin (\theta) <1.
\end{align}
From (\ref{WG1}), it follows that we can find such an $f$ satisfying the bound
\begin{align} 
f \leq L/K \sin(\theta) e^{-s}
\end{align}
The higher regularity of $f$ then follows directly. This completes the proof. 
\end{proof}

\begin{definition} \label{WingMapDefinition}
The map $\mathcal{W}_1$ being already defined in Proposition \ref{WingGeometry} above, we set
\begin{align} \notag
\mathcal{W}_2 (x, s) : = \mathfrak{R}_y \mathcal{W}_1 (x, s).
\end{align}
\end{definition}
\begin{definition}
Let $\pi_{\mathcal{S}}: \mathcal{S} \rightarrow \Reals^2$ be the projection onto the $\{ z = 0\}$ plane and set \begin{align} \notag
D[\mathcal{W}] : = \Omega \cap \{ x^2 + s^2 \geq \epsilon_0/2 \},
\end{align}
where $\epsilon_0 > 0  $ is as in Proposition \ref{ScherkNearAxisGraph}. The map  $\mathcal{W}: D[\mathcal{W}] \rightarrow \Reals^3$ is then determined by   the following requirements:
\begin{enumerate}

\item \label{FirstQuadrantRule} For $(x, s) \in D[\mathcal{W}]$ with $s  \geq 0$ it holds that 
\begin{align} \notag
\mathcal{W}(x, s)  = (1 - \psi(s))\pi^{-1}_{\mathcal{S}} (x, s) + \psi(s) \mathcal{W}_{1}(x, s)
\end{align}
where $\psi(s)$ is the cutoff function given by $\psi(s) : = \psi_0[10,  11](s)$.

\item It holds that 
\begin{align} \notag
\mathcal{W}(x, -s ) = \mathfrak{R}_y \mathcal{W}(x, s),
\end{align}
\end{enumerate}

\end{definition}

\begin{proposition}
The following statements hold:
\begin{enumerate}
\item There is a  constant $\bar{\theta} > 0$  so that for $\theta \in (0, \bar{\theta})$, the map $\mathcal{W}$ is a  minimal immersion.

\item $\mathcal{S} \setminus \mathcal{W}$ is contained within a tubular neighborhood of the z axis of radius $\epsilon_0$.
\end{enumerate}

\end{proposition}

\begin{remark}
The reader should be aware that in most places in this article we will identify the maps $\mathcal{W}$, $\mathcal{W}_i$, $\mathcal{C}$ and $\mathcal{C}_0$ with their images. In places where we need to make a distinction, it will be done explicitly. 
\end{remark}
\subsection{Convergence to $\{z = 0 \}$ at  $\theta = 0$}
As $\theta$ tends to $0$, the surfaces $\mathcal{S}$ converge to the plane $\{ z = 0\}$ away from the origin (see Figure \ref{ScherkCollapse}), although the convergence is not smooth. However, the  failure to converge smoothly to zero is due entirely to the affine term $h_{\mathcal{S}}$ in Proposition \ref{WingGeometry} (\ref{WingMap1Def}). That is, modulo vertical translations the convergence is smooth on compact subsets and the harmonic function describing the  linearization is computed below:

\begin{proposition}  \label{SmoothConvergenceToPlane}
The sets $\tilde{\Sigma} : = \Sigma - h_{\mathcal{S}} \cap \{ z \geq 0\}$ converge smoothly to the plane $\{ z= 0 \}$  on compact sets. Let $\dot{f}_{\mathcal{S}}$ denote the normal velocity $\left.\partial_\theta \tilde{\Sigma} \right|_{\theta = 0} \cdot e_z$. Then $\dot{f}_{\mathcal{S}}$ is the harmonic function
\begin{align} \notag
\dot{f}_{\mathcal{S}} (x, y) = \log\left(  \cosh(y) - \cos(x)\right),
\end{align}
regular away from the set $\{ ( 2\pi k, 0, 0 ) : k \in \mathbb{N}\}$.
\end{proposition}

\begin{proof}
We wish to compute the limit of  $(x, y, \sin(\theta) (z - h_{\mathcal{S}}))$	for $(x, y, z) \in \Sigma$. Note that such a point satisfies
\begin{align} \notag
\cos(x) & = \cos^2(\theta) \cosh(y/\cos(\theta)) - \sin^2(\theta) \cosh (z + h/\sin(\theta)) \\ \notag
& = \cos^2(\theta) \cosh(y/\cos(\theta)) - \sin^2(\theta) \left(\exp{z + \log (\cot^2(\theta))} \right) \\ \notag
& \quad  - \sin^2(\theta) \left(\exp{ - z -  \log (\cot^2(\theta))} \right) \\ \notag
& = \cos^2(\theta) \cosh(y/\cos(\theta) )- \cos^2(\theta) e^z - \sin^2 (\theta)\tan^2(\theta) e^{- z}.
\end{align}
At $\theta = 0$, we then get
\begin{align} \notag
\cos(x) = \cosh(y) - e^z.
\end{align}
Solving for $z$ then gives the claim. 
\end{proof}

\subsection{Scherk towers are close to large pieces of small catenoids near the $z$-axis} \label{ScherkNearAxis}

We  now  provide a quantitative version of statement \ref{CloseToCatenoid} in Section \ref{ScherkTowers} below:

\begin{definition} \label{UnderlyingScherk}
The map $\mathcal{C}_0 : \Omega \rightarrow \Reals^3$ is given by
\begin{align}
\mathcal{C}_0 (x, s) = 2 \kappa_{\pi/2} (x, s) = 2 \cosh(s) e_r (x) + 2 s e_z
\end{align}
\end{definition}

\begin{proposition} \label{ScherkNearAxisGraph}
There are $\bar{\theta} > 0$ and $\epsilon_0 > 0$  so that:  Given  $\theta \in (0, \bar{\theta})$, set
\begin{align}
 D[\mathcal{C}] : = \Omega \cap \{ |s| \leq \mathrm{arcosh}(\epsilon_0 \theta)\}.
\end{align}
Then there is   a function $f_\mathcal{C}:  D[\mathcal{C}]   \rightarrow R$ such that:
\begin{enumerate}
\item $f_\mathcal{C}$ satisfies the estimate
\begin{align}\notag
\|f_{\mathcal{C}} : C^{k, \alpha} ( D[\mathcal{C}] , \cosh^2 (s)) \| \leq C \sin^2(\theta).
\end{align}
\item Let $\mathcal{C}:  D[\mathcal{C}] \rightarrow \Reals^3$ be  the map given by
\begin{align} \notag
\mathcal{C} : = \sin (\theta) \mathcal{C}_0(x, s) + \sin (\theta) f_\mathcal{C} \nu[\mathcal{C}_0] 
\end{align}
Then $\mathcal{C}$ maps $ D[\mathcal{C}] $ into $\mathcal{S}$
\item The surface $\mathcal{S} \setminus  \mathcal{C} $ is contained outside of a tubular neighborhood of radius $\epsilon_0/2$ about the $z$-axis. 
\end{enumerate}
\end{proposition}

\begin{proof}
Set 
\begin{align}
F(x, y, z) : = \cos (\sin (\theta)) - \cos^2(\theta) \cosh (\tan(\theta) y) + \sin^2 (\theta) \cosh (z).
\end{align}
Then $\sin^{-1}(\theta) \Sigma$ is the zero set for $F$. Considering the Taylor expansions of $\cos(t)$ and $\cosh(t)$ gives 
\begin{align}
F(x, y, z) & = \left( 1 - \frac{\sin^2(\theta)}{2} x^2 +  O (\theta^4 x^4)\right) - \cos^{2}(\theta) \left( 1 + \frac{\tan^2 (\theta)}{2} y^2 + O (\theta^4 y^4)\right) \\ \notag
& \quad + \sin^2(\theta) \cosh (z) \\ \notag
& =  \sin^{2}(\theta) \left( 1 - \frac{x^2 + y^2}{2} + \cosh(z)\right) + O (\theta^4 x^4 ) + O( \theta^4 y^4)  \\ \notag
& : = \sin^{2}(\theta) F_0 (x, y, z) + R
\end{align}
where both $F_0$ and $R$ are defined implicitly above. For $(x, s) \in \Reals^2$ and $f \in R$, we set
\begin{align} \notag
  (x_f(x, s), y_f(x, s), z_f(x, s)) & : =  \mathcal{C}_0(x, s) + f \nu[\mathcal{C}_0] (x, s).  
\end{align}
The function $\Phi: \Reals^2 \times \Reals \rightarrow \Reals$ is then given by: 
\begin{align}
 \Phi(x, s, f) : = F(x_f, y_f, z_f) / | \nabla F(x, y, z)|.
\end{align}
Assume that $|f| \leq \delta$. Then for $\delta > 0 $ sufficiently small, we have that 
\begin{align}
\| \partial^{(k)}_f \Phi (x, s, -)  \|_{j, \alpha} \leq C(j, \alpha), \quad |\partial_f \Phi |> \frac{1}{2}.
\end{align}
where above norm taken with respect to the $(x, s)$ variables. Moreover, choosing $s \leq \mathrm{arccosh}(\delta'/\theta)$ and $\delta' $ sufficiently small, we can arrange for 
\begin{align}
\|\Phi (x, s, -) \|_{j, \alpha} \leq C \theta^2 \cosh^2 (s) \leq A
\end{align}
for arbitrary $A > 0$. The claim then immediately follows from Proposition \ref{IFTHandyVersion}
\end{proof}

\begin{corollary} \label{CatQuantities}
The following estimates hold: 
\begin{enumerate}
\item $ \| g[\mathcal{C}] - g[\mathcal{C}_0] : C^{j, \alpha}(D[\mathcal{C}], \cosh^3(s))\|  \leq C \theta^2 $
\item $\| |A|[\mathcal{C}] - |A| [\mathcal{C}_0] : C^{j, \alpha}(D[\mathcal{C}], 1))\|  \leq C \theta^2$
\item $\| \nu[\mathcal{C}] - \nu[\mathcal{C}_0] : C^{j, \alpha}(D[\mathcal{C}], \cosh(s))\|  \leq C \theta^2$
\end{enumerate}
\end{corollary}

\begin{proof}
All estimates are direct consequences of Proposition \ref{ScherkNearAxisGraph} and Proposition \ref{HQEstimates} and the fact the $g$, $|A|$ and $\nu$ are homogeneous degree $2$, $-1 $  and $0$ quantities, respectively. 
\end{proof}

\section{The stability operator on the catenoid} \label{StabilityOperatorOnCatenoid}

Let $\mathcal{L}[\mathcal{C}_0]$ be the stability operator for the immersion $\mathcal{C}_0$ given in Definition \ref{UnderlyingScherk}. In this section we study the linear problem
\begin{align} \label{CatenoidLinearProblem}
\mathcal{L}[\mathcal{C}_0] v = E
\end{align}
when the function $E$ lies in exponentially weighted H{\"o}lder spaces.  Recall that $\mathcal{C}_0$  conformally parametrizes a catenoid which closely models the geometry of the Scherk surfaces $\mathcal{S}$ near the origin--this is precisely recorded in Proposition \ref{ScherkNearAxisGraph}.  

\begin{proposition} \label{CatenoidGaussMap}
The Gauss map $\nu[\mathcal{C}_0]$ of $\mathcal{C}_0$ is a conformal diffeomorphism of $\mathcal{C}_0$ onto the unit sphere $\mathbb{S}^2$ minus the north and south pole. The conformal factor is $|A[\mathcal{C}_0]|^2/2$. In particular, equation (\ref{CatenoidLinearProblem}) is equivalent to 
\begin{align}
\left(\Delta_{\mathbb{S}^2} + 2 \right) v = 2 E/|A[\mathcal{C}_0]|
\end{align}
\end{proposition}

\begin{proposition} \label{SphereOperatorKernel}
The kernel of the operator $\Delta_{\mathbb{S}^2} + 2$ on the unit  sphere $\mathbb{S}^2$ is three-dimensional and spanned by the coordinate functions $x$, $y$ and $z$.
\end{proposition}

\begin{definition} \label{UnderlyingCatenoidSpaces}
For $\alpha, \gamma  \in (1/2, 1)$ and $\ell > 0$ fixed, we set 
\begin{align} \notag
\mathcal{X}^{0}[\mathcal{C}_0] & : = C_0^{0, \alpha}(\Omega_{\leq \ell}, \cosh^{\gamma - 2}(s)) \\ \notag \mathcal{X}^{2}[\mathcal{C}_0] & : = C^{2, \alpha}( \Omega, \cosh^\gamma(s)) \cap C^{k, \alpha}(\Omega_{\geq \ell + 1}, 1 )
\end{align}
We also let $\mathcal{X}_{\perp}^{0}[\mathcal{C}_0] \subset \mathcal{X}^{0}[\mathcal{C}_0] $ denote the subspace of functions satisfying the following orthogonality conditions:
\begin{align} \notag
\int_{\Omega} E d \mu [\mathcal{C}_0] = \int_{\Omega} E  \phi d \mu [\mathcal{C}_0],
\end{align}
where $\phi$ above denotes the pullback to $\Omega$ under the Gauss map of $\mathcal{C}_0$ of the coordinate functions $x$, $y$, $z$ on $\mathbb{S}^2$.
\end{definition}

 The main result is then:
\begin{proposition} \label{CatenoidInverse}
There is a bounded linear map 
\begin{align} \notag
\mathcal{R}[\mathcal{C}_0, -]:  \mathcal{X}_{\perp}^{0}[\mathcal{C}_0]  \rightarrow \mathcal{X}^{2}[\mathcal{C}_0]
 \end{align}
so that:

\begin{align} \notag
\mathcal{L}[\mathcal{C}_0]\mathcal{R}[\mathcal{C}_0,  E ] = E.
\end{align}
\end{proposition}
Before proving Proposition \ref{CatenoidInverse}, we first record a few useful observations. 

\begin{lemma} \label{UnderlyingCatenoidQuantities}
The conformal factor $\varrho[\mathcal{C}_0] (s)$ for the conformal immersion $\mathcal{C}_0$ is given by
\begin{align} \notag
\varrho[\mathcal{C}_0] (s) = 2 \cosh (s).
\end{align}
The square length of the second fundamental form  $|A[\mathcal{C}_0]|^2$ is given by
\begin{align} \notag
|A[\mathcal{C}_0]|^2(s) = \frac{1}{2} \cosh^{-2} (s). 
\end{align}
\end{lemma}

From Lemma \ref{UnderlyingCatenoidQuantities} it then follows directly that we can write
\begin{align} \notag
\mathcal{L}[\mathcal{C}_0]  & = \frac{1}{4}\cosh^{-2}(s) \left\{ \Delta_{\Omega} + 2 \cosh^{-2}(s)\right\} \\ \notag
&  = :  \frac{1}{4}\cosh^{-2} (s) \tilde{\mathcal{L}}[\mathcal{C}_{0}],
\end{align}
so that the linear problem (\ref{CatenoidLinearProblem}) can be written in the equivalent form
\begin{align}
\tilde{\mathcal{L}}[\mathcal{C}_{0}]v = 4 \cosh^{2}(s) E  = :  \tilde{E}. 
\end{align}
Note that by definition it holds that 
\begin{align}
\|\tilde{E} : C^{0, \alpha}(\Omega, \cosh^{\gamma}(s))\| = \|E : \mathcal{X}^{0}[\mathcal{C}_0]\|.
\end{align}
We then have

\begin{lemma}\label{SphereLiftedProblem}
Equation (\ref{CatenoidLinearProblem}) can be equivalently stated on the sphere as:
\begin{align} \notag
\left(\Delta_{S^2} + 2 \right) v = 2 E/|A^2[\mathcal{C}_0]|
\end{align}
where we have identified functions with there lifts to $\mathbb{S}^2$ under the Gauss map of $\mathcal{C}_0$.
\end{lemma}

\begin{proof}[Proof of Proposition \ref{CatenoidInverse}]

We write $E$ in the orthogonal decomposition:
\begin{align} \notag
E (x,s) : =  \bar{E}(s) + \mathring{E}(x, s)
\end{align}
where
\begin{align} \notag
\bar{E}(s) : = \frac{1}{2\pi} \int_{-\pi}^\pi E(x, s) dx.
\end{align}
 denotes the meridian average of $E$.  Let $x^*$, $y^*$ and $z^*$ denote the pullbacks to $\Omega$ under $\nu[\mathcal{C}_0]$ of $x$, $y$ and $z$, respectively. We then have directly that 
  \begin{align} \notag
z^*(x, s) & =  - \tanh(s), \\ \notag
x^*(x, s) & =  \sin(x)\cosh^{-1}(s), \\ \notag
 y^*(x, s) & =  \cos(x)\cosh^{-1}(s).
\end{align}
It then follows directly that $\bar{E}(s)$ is automatically $d \mu[\mathcal{C}_0] = 4 \cosh^2(s) dx ds$-orthogonal to $x^*$ and $y^*$ independent of any orthogonality assumptions on $E$, and that $\mathring{E}$ is then likewise necessarily orthogonal to $z^*$. From the orthogonality condition on $E$ we then additionally get that 
\begin{align}
\int_\Omega \bar{E}  z^* d \mu[\mathcal{C}_0] &  =  \int_\Omega E  z^*   d \mu[\mathcal{C}_0] = 0 \\ \notag
\int_{\Omega} \mathring{E} x^* d \mu[\mathcal{C}_0] & = \int_\Omega E x^* d \mu[\mathcal{C}_0] = 0  \\ \notag 
\int_{\Omega} \mathring{E} y^* d \mu[\mathcal{C}_0] & = \int_\Omega E y^* d \mu[\mathcal{C}_0] = 0
\end{align}
  As a preliminary step, we first improve the asymptotic behavior of the error term by solving the flat laplacian for $\mathring{F}= 4 \cosh^2 (s) \mathring{E} $. That, is we set
\begin{align} \notag
\mathring{v}_0  : = \mathcal{R}_1[\mathring{F}], \quad  \mathring{E}_{1} & : = \mathring{E} - \mathcal{L}[\mathcal{C}_0] \mathring{v}_0  =  - 2 \cosh^{- 4}(s) \mathring{v}_0
\end{align}
Thus, the function $\mathring{v}_0$ satisfies the equation
\begin{align}
\Delta_{\Omega} \mathring{v}_0 = \mathring{F}.
\end{align}
Additionally, the function $\mathring{v}_0$ has zero average along meridians by Proposition \ref{LaplaceZeroMeridianInverse} so that $\mathring{E}_1$ also has zero average along meridians. This then immediately gives that $\mathring{E}_1$ satisfies the orthogonality conditions in Definition \ref{UnderlyingCatenoidSpaces}. Then we have directly that 

It then immediately follows that 
\begin{align} \notag
\int_{\Omega} \mathring{E}_1  z^*(x, s) d \mu[\mathcal{C}_0]  &= 4 \int_\Omega \mathring{E}_1 (x,s) z^*(x, s) \cosh^2 (s) dx ds \\ \notag
&  =  4\int_{-\infty}^\infty z^*(s) \cosh^2(s) \left(\int_{-\pi}^\pi  \mathring{E}_1 (x,s) dx \right) ds \\ \notag
& = 0.
\end{align}
We then have
\begin{align} \notag
\int_{\Omega} \mathring{E}_{1} x^* d \mu[\mathcal{C}_0] & = \int_{\Omega} \mathring{E} x^*  d \mu[\mathcal{C}_0] - \int_{\Omega} \left(\mathcal{L}[\mathcal{C}_0] v_0\right) x^*  d \mu[\mathcal{C}_0] \\ \notag
& = 0   -  \int_{\Omega} \left(\tilde{\mathcal{L}}[\mathcal{C}_0] v_0\right) x^* \\ \notag
& = 0 + \lim_{{N \rightarrow \infty}}\int_{- \pi}^\pi \partial_s v_{0} (x, N) x^*(x, N) -v_0(x, N) \left(\partial_s x^* \right) (x, N) \\ \notag
& = 0.
\end{align}
The first term on the right hand side of the first line above is zero due to the initial orthogonality  property of  $\mathring{E}$ and the second term is zero by integrating by parts and considering the subexponential growth rate of $\mathring{v}_0$ and the exponential decay of $x^*$. Thus, we conclude that $\mathring{E}_1$ is orthogonal to $x^*$. A similar argument shows that $\mathring{E}_1$ is orthogonal to $y^*$ as well.  Moreover, $\mathring{E}_1$ satisfies the improved weighted estimate:
\begin{align} \notag
\| \mathring{E}_1 : C^{0, \alpha}(\Omega, \cosh^{\gamma - 4} (s))\| \leq C \| E : \mathcal{X}^{0}[\mathcal{C}_0]\|.
\end{align}
 To solve for the meridian average $\bar{E}$, we simply set
\begin{align} \notag
\bar{v}(s) & = 4 \int_{s}^\infty \int_{s'}^\infty \bar{E} (s'') \cosh^2(s'') ds'' \\ \notag
\end{align}
The orthogonality conditions on $E$ imply that 
\begin{align}
\int_{-\infty}^\infty\bar{E}(s') \cosh^2(s') ds' = 0.
\end{align}
Thus, as in the proof of Proposition \ref{LaplaceZeroMeridianInverse}, we get that
\begin{align}
\bar{v} (s) = \bar{v}(s) & = 4 \int_{-\infty}^s \int_{-\infty}^s  \bar{E} (s'') \cosh^2(s'') ds''.
\end{align}
In order to solve for $ \mathring{E}_1$, we note first that
\begin{align} \notag
\int_{S^2} \left( 2 E_1/| A[\mathcal{C}_0]^2 |\right) x d\mu[S^2] = \int_{\mathcal{C}_0} \mathring{E}_1 x^* d \mu[\mathcal{C}_0] = 0, 
\end{align}
and similarly for integrating against $y$ and $z$. Thus Proposition \ref{SphereOperatorKernel} and standard theory give a function  $\mathring{v}_1$ on $S^2$ solving the equation
\begin{align}
\left(\Delta_{\mathbb{S}^2} + 2 \right) \mathring{v}_1 = 2 \mathring{E}_1/| A[\mathcal{C}_0]^2 |
\end{align}
 and satisfying the estimate
\begin{align}
\| \mathring{v}_1 : W^{2, 2}(S^2)\| \leq C \| 2 \mathring{E}_1 / |A[\mathcal{C}_0]|^2: L^2(S^2) \|.
\end{align}
It remains to produce an estimate for the $L^2$-norm of the right hand side. To do this, we write
\begin{align} \notag
& \int_{S^2} \left( 2 \mathring{E}_1/ |A[\mathcal{C}_0]|^2 \right)^2 d \mu[S^2]   = \int_{\Omega} \left( \mathring{E}_1^2/ |A[\mathcal{C}_0]|^2  \right) d \mu[\mathcal{C}_0] \\  \notag
& \leq  \| \mathring{E}_1 : C^{0, \alpha} (\Omega, \cosh^{\gamma - 4})(s) \|^2 \int_{\Omega} \left( \cosh^{2 \gamma - 4}(s) \right) \left( \cosh^2 (s) d \mu[\Omega]\right) \\ \notag
& \leq C  \| E_1 : C^{0, \alpha} (\Omega, \cosh^{\gamma - 4})(s) \|^2
\end{align}
where the last line above follows from the fact that $ \gamma \in (0, 1)$. It then immediately follows that 
\begin{align}
\| \mathring{v}_1 : W^{2, 2}(S^2)\| \leq C   \| E : \mathcal{X}^{0}[\mathcal{C}_0]\|.
\end{align}
We abuse notation slightly by identifying $\mathring{v}_1$ with its pull back to $\Omega$. We then have that 
\begin{align} \notag
\sup_{\Omega} |\mathring{v}_1| \leq C\| E : \mathcal{X}^{0}[\mathcal{C}_0]\|.
\end{align}
Standard elliptic theory then gives the higher estimate 
\begin{align} \notag
\| \mathring{v}_1 : C^{2, \alpha} (\Omega, 1)\| \leq C \| E : \mathcal{X}^{0}[\mathcal{C}_0]\|.
\end{align}
We conclude by setting $\mathring{v} : = \mathring{v}_0 + \mathring{v}_1$, which completes the proof in the case that $E$ has zero average along meridians.\end{proof}

\section{Bending Scherk towers around circles} \label{BendingMapsSection}

We wish to use the surfaces $\mathcal{S}$ to construct minimal surfaces with a discrete rotational symmetry in place of  a translational invariance. We do this  essentially by deforming each surface by a diffeomorphism which introduces small constant curvature to the axis of periodicity.
\subsection{The bending maps and their properties}

\begin{definition} \label{BendingMap} 
The map $B: \Reals^3 \rightarrow \Reals^3$ is given below:
\begin{align} \notag
B(x, y, z) : = (\tau^{-1} + y) \sin(\tau  x) e_x + (\tau^{-1} + y)( \cos(\tau  x)  - 1)e_y + z e_z.
\end{align}
\end{definition}

In Proposition \ref{BendingMapProps} and Definition \ref{ModifiedBendingMaps} below, the reader may wish to recall Definitions \ref{IsometryDefs} and \ref{IsometryGroupDefs}.
\begin{proposition}\label{BendingMapProps}
 The map  $B$ has the following properties:
\begin{enumerate}
\item It holds that 
\begin{align} \notag
B \circ \mathfrak{T}_t = \mathfrak{T}^*_t \circ B,  \quad B \circ \mathfrak{R}_x = \mathfrak{R}_x \circ B, \quad B \circ \mathfrak{R}_z = \mathfrak{R}_z \circ B
\end{align}
for $t \in \Reals$. In particular, the maps $B$ are $\mathfrak{G}$ equivariant. 
\item The maps $B$ depend smoothly on $\tau$ on compact subsets and agree with the identity at $\tau = 0$.
\item The linearization of $B$ at the origin is the identity. 
\end{enumerate}

\end{proposition}

For technical reasons, it is convenient to modify the maps $B$, preserving $\frak{G}$ equivariance and so that they agree with the identity map in small neighborhoods about the origin.

\begin{definition} \label{ModifiedBendingMaps}
We  let $\tilde{B}$ be maps determined as follows: Recall the constant $\epsilon_0 > 0$ in Proposition \ref{ScherkNearAxisGraph}. Then:
\begin{enumerate}
\item The map $\tilde{B}$ is $\mathfrak{G}$-equivariant.
\item On compact subsets the maps $\tilde{B}$  depend smoothly  on $\tau$ and agree with the identity  map when $\tau = 0$.
\item On the set $\{ |x| \leq 2 \pi, y^2 + z^2 \leq 4 \epsilon_0^2\}$, the maps $\tilde{B}$ agree with the identity map for all $\tau$.
\item On the set $\{ |x| \leq 2 \pi, y^2 + z^2 \geq 16 \epsilon_0^2\}$, the maps $\tilde{B}$  agree with $B$.

\end{enumerate}
 
\end{definition}

\section{Matching bent Scherk towers with catenoidal ends} \label{MatchingSection}

\begin{definition}\label{BetaiDefs}
We set 
\begin{align} \notag
\beta_{\mathcal{S}, 1} : = \theta, \quad \beta_{\mathcal{S}, 2} : = \pi - \theta.
\end{align}

\end{definition}
Note that the asymptotic planes for the wing $\mathcal{W}_i$ are then $H^+[\beta_{\mathcal{S}, i}, h_{\mathcal{S}}]$.

Before continuing the reader may wish to recall the definition of the functions $\tilde{\varrho}[\beta]$ in Proposition \ref{RenormKappaBasicProps} and the Definition of $f_\mathcal{W}$ in Proposition \ref{WingGeometry}.  In Definition \ref{BentWingMaps} below, we construct immersions by adding a weighted normal graph of the function $f_\mathcal{W}$ to the catenoidal ends $\tilde{\kappa}[\beta]$. We do this in a  separate step before defining the initial surfaces because there are some technicalities involved in properly estimating their mean curvature, which are simpler to treat independently. 

\begin{definition} \label{BentWingMaps}
Given $b, d \in \Reals$  the maps $\mathcal{K}_i[b, d] (x, s): \Omega^+_{\geq 1} \rightarrow \Reals^3$, $i = 1, 2$ are given by:
\begin{align} \notag
\mathcal{K}_i[b, d] (x, s) : = \tilde{\kappa}[\beta_{\mathcal{S}, i} + d] (x, s)+ (h_{\mathcal{S}} + b) e_z + \tilde{\varrho}[\beta_{\mathcal{S}, i} + d] f_\mathcal{W}(x, s)\nu[\tilde{\kappa}[\beta_{\mathcal{S}, i} + d] ] (x, s)
\end{align}
\end{definition}

\begin{proposition} \label{BentWingGeometry}
 Set $\mathcal{K}_i = \mathcal{K}_i[b, d]$, $\tilde{\varrho}_i = \tilde{\varrho}[\beta_{\mathcal{S}, i} + d]$ . Then there is $\epsilon > 0$  so that for $ \tau, \theta \in (0, \epsilon]$ we have: 

\begin{enumerate}
\item \label{WeightedRegularity}The maps $\mathcal{K}_i$ are smooth immersions, depending smoothly on $\tau$, $\theta$, $b$ and $d$. 
\item  \label{BentWingMC} There is   $C = C(j, \alpha)$ so that:
\begin{align} \notag
 \| \tilde{\varrho}_i  H[\mathcal{K}_i] : C^{j, \alpha}(\Omega^+_{\geq 1}, \cosh^{-1}(s)) \| \leq C\tau \theta.
\end{align}

\item  \label{BentWingSO} It holds that
\begin{align} \notag
\| \tilde{\varrho}^2_i \mathcal{L}[\mathcal{K}_i] - \mathcal{L}[\mathcal{W}_i]: C^{j, \alpha} (\Omega^+_{\geq 1}, 1) \| \leq C \tau,
\end{align}
where the norm above is applied to the coefficients of the operator $\tilde{\varrho}^2_i \mathcal{L}[\mathcal{K}_i] - \mathcal{L}[\mathcal{W}_i]$.
\end{enumerate}
\end{proposition}

\begin{definition} 
Set $\tilde{\kappa} = \tilde{\kappa}[\beta]$, $\tilde{\varrho} : = \tilde{\varrho}[\beta]$,  $\beta \in \mathbb{R}$. Then for $k \in \mathbb{N}$ we set
\begin{align} \notag
\tilde{M}_{\kappa }^{(k)}[\tau, f] (x, s) : = \left[\tilde{\varrho}^{-1}J^{(k)}[ \tilde{\kappa} +  \tilde{\varrho} f \nu[\tilde{\kappa}] ](x, s) \right]_{\{ \tilde{e}^\kappa\}} 
\end{align}
where the $k$-jet quantity $J^{(k)}[-]$ is given in Definition \ref{PhiJets}. That is, at each point $(x, s)$,  $\tilde{M}_{\kappa }^{(k)}[\tau, f] (x, s)$ records the components of $\tilde{\varrho}^{-1}J^{(k)}[\tilde{\kappa} + f \nu[\tilde{\kappa}] ](x, s)$ expressed in the basis $\{\tilde{e}^\kappa\}(x, s) : =  \{e^\kappa \} (\tau x, \tau s)$.
\end{definition}
The reader may wish to recall the definition of the orthonormal basis $\{ e^\kappa_i\}$ in Proposition \ref{TechnicalKappaProps} (\ref{KappaFrame}).
\begin{lemma} \label{BentWingJets}
There is a constant $ C =  C (k, j, \alpha)$ independent of $x$ and $s$ so that:
\begin{enumerate} \label{WingJetEstimates}
\item $\| \tilde{M}_{\kappa }^{(k)}[\tau, f] - \tilde{M}^{(k)}[\tau, 0] \|_{j, \alpha} \leq C \| f\|_{j, \alpha}$

\item  $\| \tilde{M}_{\kappa }^{(k)}[\tau, 0] -\tilde{M}_{\kappa }^{(k)}[0, 0]  \|_{j, \alpha} \leq C \tau$. 
\end{enumerate}
\end{lemma}

\begin{proof}
Recalling the quantities $a$ and $b$ given in Lemma \ref{TechnicalKappaProps} (\ref{abProps}) and computing directly gives
\begin{align} \notag
& ( \tilde{\kappa} +  \tilde{\varrho} f \nu[\tilde{\kappa}] )_x = \tilde{\kappa}_x + \tilde{\varrho} f_x \nu + \tilde{\varrho} f \nu_x = \tilde{\varrho} (1 - \tau b f) \tilde{e}_1^\kappa + \tilde{\varrho} f_x \nu. \\ \notag
& ( \tilde{\kappa} +  \tilde{\varrho} f \nu[\tilde{\kappa}])_s = \tilde{\kappa}_s + \tilde{\varrho}' f \nu + \tilde{\varrho} f_s + \tilde{\varrho} f \nu_s =  \tilde{\varrho} (1 + \tau b f)  \tilde{e}_2^\kappa + \tilde{\varrho} f_s \tilde{e}_3^\kappa + \tilde{\varrho}' f \tilde{e}^\kappa_3.
\end{align}
Thus, we have
\begin{align} \notag
 \tilde{M}_{\kappa }^{(1)}[\tau, f]   = 
\left( \begin{array} {cc}
1 - \tau b f & 0 \\
0 & 1 + \tau b f \\ 
f_x & f_s + \tau a f
\end{array} \right).
\end{align}
The estimates in (\ref{WingJetEstimates}) then follow directly in the case $k = 1$ and $j$ is arbitrary.  For higher $k$, we proceed by induction and derive an explicit expression which relates $\tilde{M}_\kappa^{(k + 1)}$ with $\tilde{M}_\kappa^{(k )}$.

Note that the derivative matrix of the frame $\tilde{e}^\kappa$ is given by $\tilde{T}^\kappa(s) : = \tau T^\kappa (\tau \, s)$ (Recall Lemma \ref{TechnicalKappaProps}). Now, let $V$ be a vector in $J^{(k)}(\tilde{\kappa} + \tilde{\varrho} f \nu)$. That is, $V$ is of the form
\begin{align} \notag
V : = \partial^{\alpha}_{x} \partial^{\beta}_s \left( \tilde{\kappa}+  \tilde{\varrho} f \nu[\tilde{\kappa}]\right)
\end{align}
where  $\partial^{\alpha}_{x}$ and $\partial^{\beta}_s$ denote pure derivatives in $x$ and $s$ of order $\alpha$ and $\beta$ respectively  and so that $\alpha + \beta = k$.  We can then write
\begin{align}
V : = V^i \tilde{e}^\kappa_i  
\end{align}
where the coefficients $V_i$ belong to the matrix $\tilde{\varrho} \tilde{M}_\kappa^{(k)}[\tau, f]$.
Every vector in $J^{(k + 1)}[\tilde{\kappa} +   \tilde{\varrho}f \nu]$ is then one of the following forms for some $V$
\begin{align}
\partial_x V, \quad \partial_s V. 
\end{align}
In the  first case, we can write
\begin{align}
\partial_x V  & = ( \partial_x V^i ) \tilde{e}^\kappa_i + V^i (\partial_x \tilde{e}^{\kappa}_i) \\ \notag
& =  ( \partial_x V^i ) \tilde{e}^\kappa_i + V^i  (\tilde{T}^{\kappa}_x)^{ j}_{ i} \tilde{e}^\kappa_j.
\end{align}
Repeating the argument for $s$ gives:
\begin{align} \notag
\partial_s V =   (\partial_s V^i ) \tilde{e}^\kappa_i + V^i  (\tilde{T}_s^\kappa)^j_{ i} \tilde{e}^\kappa_j
\end{align}

From this it follows that:
\begin{align} \notag
 \tilde{M}_{\kappa }^{(k + 1)}[\tau, f]  & = ( \nabla \tilde{\varrho}/\tilde{\varrho}) \otimes \tilde{M}_{\kappa }^{(k )}[\tau, f]+ \nabla J^{(k)} +\tilde{T}^\kappa *  \tilde{M}_{\kappa }^{(k)}[\tau, f] . 
\end{align}
Where above $*$ denotes a contraction of components of  $\tilde{T}^\kappa$ with $ \tilde{M}_{\kappa }^{(k)}[\tau, f] $.
Since $\nabla \tilde{\varrho} = (0,  \tau \, a)$, the claim then follows  from Lemma \ref{TechnicalKappaProps} and induction on  $k$.
\end{proof}

\begin{proof} [Proof of Proposition \ref{BentWingGeometry}]
Set
\begin{align}
 \underline{J}[f] & : = (J^{(1)}[\tilde{\kappa} +  \tilde{\varrho} f \nu], J^{(2)}[\tilde{\kappa} +  \tilde{\varrho} f \nu]) \\ \notag
 \underline{\tilde{M}}_\kappa[\tau, f] & : = (\tilde{M}^{(1)}_\kappa[\tau, f] , \tilde{M}^{(2)}_\kappa[\tau, f] )
\end{align}
 where  $\tilde{\kappa} = \tilde{\kappa}[\beta_{\mathcal{S}, i} + d]$ $\tilde{\varrho} = \tilde{\varrho}[\beta_{\mathcal{S}, i} + d].$ Observe that  Definition \ref{BentWingMaps} gives that
 \begin{align}
  \underline{J}[f_\mathcal{W}]  = J[ \mathcal{K}_i[b, d]].
 \end{align}
  Since the coefficients of the stability operator are homogeneous degree $-2$ regular quantities which are invariant under rotations of $\Reals^3$, we can write
\begin{align} \notag
\mathcal{L}[\underline{J}] = \tilde{\varrho}^{- 2} \mathcal{L}[ \tilde{\varrho}^{-1}\underline{J}] = \tilde{\varrho}^{-2} \mathcal{L}[(\underline{\tilde{M}}_\kappa[\tau, f])^i \tilde{e}^\kappa_i].
\end{align}
Claim (\ref{BentWingSO}) then immediately follows from Lemma \ref{BentWingJets} and Proposition \ref{HQEstimates}.

 We now prove (\ref{BentWingMC}). Let $F(J, R) : \mathcal{J} \times \mathcal{J} \rightarrow \Reals$ (Recall Definition \ref{BentWingMC} for $\mathcal{J}$) be the function given by
\begin{align} \notag
F(J, R)  : = H(J  + R)- H (J).
\end{align}
where $H: \mathcal{J} \rightarrow \Reals$ is the mean curvature function. Set
\begin{align}
 R_\kappa[\tau, f]: = \underline{\tilde{M}}_\kappa[\tau, f]  -  \underline{\tilde{M}}_\kappa[\tau, 0]. 
\end{align}
F is then a smooth function defined in a neighborhood of $\underline{M}_\kappa[0, 0]$ and it holds that 
\begin{align} \notag
F(\underline{\tilde{M}}_\kappa[0, 0], R_\kappa[0, f_\mathcal{W}]) = F(\underline{\tilde{M}}_\kappa[\tau, 0], 0) = 0 .
\end{align}
For $\| M -  \underline{\tilde{M}}_\kappa[0, 0] \|_{j, \alpha} $ and $\| R- R_\kappa[0, f_{\mathcal{W}}] \|_{j, \alpha}$ sufficiently small, smoothness of the function $F$ then gives: \begin{align} \notag
\|F(M, R) \|_{j, \alpha} \leq C \| R \|_{j, \alpha}\| M - \underline{\tilde{M}}_\kappa[0, 0] \|_{j, \alpha} + C\| R - R_\kappa[0, f_{\mathcal{W}}] \|_{j, \alpha}
\end{align}
 Writing
 \begin{align} \notag
 \tilde{\varrho} H [\mathcal{K}] = F(\underline{\tilde{M}}_\kappa[\tau, 0], R_\kappa(f_\mathcal{W}))
 \end{align}
 and using Lemma \ref{BentWingJets} then gives the claim. 
\end{proof}

\begin{definition} \label{RealWingMapsDef}
For $i = 1, 2$ , the maps $\mathcal{W}^*_i[b, d]: H^+_{\geq 1} \rightarrow \Reals^3$ are given as follows:
\begin{align} \notag
\mathcal{W}^*_i[b, d](x, s) : = (1 - \psi_0[1, 2](s)) B \circ \mathcal{W}_i(x, s) + \psi_0[1,2](s) \mathcal{K}_i[b, d](x, s).
\end{align}
\end{definition}

\begin{proposition} \label{RealWingMapProps}
There are $\bar{\tau} > 0$ and $\bar{\theta} > 0$ such that for  $\tau \in [0, \bar{\tau}), \theta \in (0, \bar{\theta})$ the following statements hold:
\begin{enumerate}
\item The maps $\mathcal{W}_i^*[b, d] $ are smooth, regular immersions depending smoothly on $\tau$ and $\theta$, $b$ and $d$.

\item It holds that $\mathcal{W}_i^*[b, d] = \mathcal{K}_i[b, d]$ for $s \geq 2$ and  $\mathcal{W}_i^*[b, d] = B \circ \mathcal{W}_i$ for $s \leq 1$.
 \end{enumerate}
 
\end{proposition}

\section{ The initial surfaces } \label{TheInitialSurfaces}

\begin{definition} \label{InitialSurfaceMap}
Let $\underline{\varphi} $ be a vector in $\Reals^4$ and write $\underline{\varphi} = (d_1,  d_2, b_1,b_2)$. Then the maps $\mathcal{Z}[\underline{\varphi}]: \mathcal{S} \rightarrow \Reals^3$ are determined as follows: 
\begin{enumerate}
\item For $p \in \mathcal{W}_i$ we have
\begin{align} \notag
\mathcal{Z}[\underline{\varphi}] (p)  : = \mathcal{W}_i^*[b_i, d_i] \circ \mathcal{W}_i^{-1} (p)
\end{align}

\item Otherwise, we take
\begin{align} \notag
\mathcal{Z}[\underline{\varphi}] : = \tilde{B} (p).
\end{align}
\end{enumerate}
\end{definition}

\begin{proposition} \label{InitialSurfaceProps}
Let $\mathcal{S}^*[\underline{\varphi}]$ be the image of $\mathcal{S}$ under $\mathcal{Z}[\underline{\varphi}]$. Then there are constants  $\bar{\tau} > 0$, $\bar{\theta} > 0$ and $\delta_0$  so that for $\tau \in [0, \bar{\tau})$, $\theta \in (0, \bar{\theta})$, and   $|\underline{\varphi}| \in [0, \delta_0)$ the following statements hold:
\begin{enumerate}
\item \label{SmoothInitialSurfaces}The surface $\mathcal{S}^*[\underline{\varphi}]$ is a smooth  regular immersed surface depending smoothly on $\tau$, $\theta$ and $\underline{\varphi}$.

\item \label{GrowthInitialSurfaces}The maps  $\mathcal{W}^*_{ i}[\underline{\varphi}] : = \mathcal{Z}[\underline{\varphi}] (\mathcal{W}_{ i})$ are  asymptotic to   catenoidal ends with a common axis and logarithmic growth  equal to $\sin(\theta + (-1)^{ i - 1} \varphi_i)/ \tau$ for $i = 1, 2$. In particular, the surface $\mathcal{S}^*[\underline{\varphi}]$ is 
embedded whenever $\varphi_1 + \varphi_2 \leq 0$ and $\tau^{-1}$ is an integer, and non-embedded otherwise.
\end{enumerate}
\end{proposition}
\begin{proof}
Statement (\ref{SmoothInitialSurfaces})  is a direct consequence of the smooth dependence on compact sets of $\tilde{B}$ and  Proposition \ref{RealWingMapProps}. Statement (\ref{GrowthInitialSurfaces}) follows from Proposition \ref{RenormKappaBasicProps}.
\end{proof}

\subsection{Graphs over the surfaces $\mathcal{S}[\theta]$}

\begin{definition} \label{ScherkFunctionSpaces}
Let $\gamma \in (1/2, 1)$ and $\alpha \in (0, 1)$ be fixed. Then a function $f: \mathcal{S} \rightarrow \Reals$ belongs to the space $\mathcal{X}^k$, $k = 0, 2$ if and only if:
\begin{enumerate}
\item $f$ belongs to the space $C^{k, \alpha}_{loc}(\mathcal{S})$. 
\item $f$ is $\mathfrak{G}$ invariant. 
\item \label{WingDecay} It holds that 
\begin{align} \notag
 I : = \|f : C^{k, \alpha} (\mathcal{W}, \cosh^{ -  \gamma} (s)) \| < \infty.
\end{align}
\item \label{CatenoidGrowth} It holds that 
\begin{align} \notag
II : = \|f : C^{k, \alpha} (\mathcal{C}, \theta^{\gamma -2 + k}\cosh^{ \gamma  - k + 2} (s)) \| < \infty.
\end{align}
\end{enumerate}
The  decay condition  (\ref{WingDecay}) above ensures that $\mathcal{X}^k$ is a Banach space in the norm $\| -: \mathcal{X}^k\|$ given by
\begin{align} \notag
\| f: \mathcal{X}^k\| : =  \max\{I, II \}.
\end{align}

\end{definition}

\begin{definition}
We let $\varrho^*: \mathcal{S} \rightarrow \Reals$  be a smooth function such that:
\begin{enumerate}
	\item It holds that 
		\begin{align} \notag
			\varrho^*\circ \mathcal{W}_{i} = \tilde{\varrho}_i, \quad i = 1, 2.
		\end{align}
		(Recall the definition of $\tilde{\varrho}_i$ in Proposition \ref{BentWingGeometry})
\item The functions converge smoothly on compact subsets of $\mathcal{W}$ to $1$ as $\tau$ approaches zero. 

\item  The functions are identically equal to $1$ on $\mathcal{C}$. 
\end{enumerate}
\end{definition}

\begin{definition} \label{InitialGraphs}
Given a function $u : \mathcal{S} \rightarrow \Reals$, we let $\mathcal{Z}[\underline{\varphi}, u]: \mathcal{S} \rightarrow \Reals^3$ be the map given as follows:
\begin{align} \notag
\mathcal{Z}[\underline{\varphi}, u] (p): = \mathcal{Z}[\underline{\varphi}] + \varrho^*[\underline{\varphi}](p) u (p) \nu[\mathcal{S}^*[\underline{\varphi}]] (p)
\end{align}
\end{definition}

\begin{proposition} \label{InitialGraphProps}
There are constants  $\bar{\tau} > 0$, $\bar{\theta} > 0$ and $\delta_0 > 0$  so that for $\tau \in [0, \bar{\tau})$, $\theta \in (0, \bar{\theta})$, and   $|\underline{\varphi}|, \| u : \mathcal{X}^2 \| \in [0, \delta_0)$,  the following statements hold:
\begin{enumerate}
\item \label{IGRegularity} The surface $\mathcal{S}^*[\underline{\varphi}, u]$ is a locally  $C^{2, \alpha}$ regular immersed surface depending smoothly on $\tau$, $\theta$, $\underline{\varphi}$ and $u$ on compact subsets of $\Reals^3$.

\item  \label{IGAsymptotics} The maps  $\mathcal{W}^*_{ i}[\underline{\varphi}, u] : = \mathcal{Z}[\underline{\varphi}, u] (\mathcal{W}_{ i})$ are  asymptotic to  catenoidal ends with a common axis and logarithmic growth equal to equal to $\sin(\theta + (-1)^{ i - 1} \varphi_i)/ \tau$ for $i = 1, 2$. In particular, the surface $\mathcal{S}^*[\underline{\varphi}, u]$ is 
embedded whenever $\varphi_1 + \varphi_2 \leq 0$ and $\tau^{-1}$ is an integer, and non-embedded otherwise.
\end{enumerate} 
\end{proposition}

\section{Mean curvature of the initial surfaces} \label{MeanCurvatureOfInitialSurfaces}

\begin{definition} \label{InitialSurfaceErrorTerm}
We let $H^*[\underline{\varphi}, u] = H[\mathcal{S}^*[\underline{\varphi}, u]]$  denote the mean curvature of $\mathcal{S}^*[\underline{\varphi}, u]$.  We will throughout abuse notation and identify $H^*$ with its pullback to $\mathcal{S}$ under $\mathcal{Z}[\underline{\varphi}, u]$. 
\end{definition}

\begin{proposition} \label{MCVariationField}
We denote the variation field \notag
\begin{align} 
\xi :  = \partial_\tau \left.   \mathcal{Z}  \right|_{\tau = 0, \underline{\varphi} = 0}
\end{align}
where $\mathcal{Z}[\underline{\varphi}]$ is evaluated at $\tau = 0$, $\underline{\varphi} = 0$. Then it holds that 
\begin{align}  \notag
\partial_\tau \left. H^* \right|_{\tau = 0, \underline{\varphi} = 0} =  \mathcal{L} \xi^\perp. 
\end{align}
where the function $\xi^\perp : = \xi \cdot \nu[\mathcal{S}]$.
\end{proposition}
\begin{proof}
The stability operator records the variation of the mean curvature under a normal perturbation. Since the surface $\mathcal{S}$ is minimal, the tangential part of the perturbation field does not  contribute to the mean curvature variation. 
\end{proof}

\begin{definition} \label{HatFunctions}
We let $\hat{u}_i: \mathcal{S} \rightarrow \Reals$, $i = 1, \ldots, 4$ be the functions determined as follows:

\begin{align} \notag
\hat{u}_i : = \left. \partial_{\varphi_i} \mathcal{Z}[\underline{\varphi}] \right|_{\tau = 0, \underline{\varphi} = 0} \cdot \nu[\mathcal{S}]
\end{align}
where above the maps $\mathcal{Z}[\underline{\varphi}]$ are all evaluated at $\tau = 0$. 
We also set 
\begin{align} \notag
\hat{w}_i = \mathcal{L} \hat{u}_i.
\end{align}
We write
\begin{align} \notag
\underline{\hat{u}} : = (\hat{u}_1, \hat{u}_2, \hat{u}_3, \hat{u}_4), \quad \underline{\hat{w}} : = (\hat{w}_1, \hat{w}_2, \hat{w}_3, \hat{w}_4)
\end{align}
and given a vector $v = (v_1, v_2, v_3, v_4) \in \Reals^4$ we abbreviate
\begin{align} \notag
v \cdot \underline{\hat{u}} : = \sum_i v_i \hat{u}_i, \quad v \cdot \underline{\hat{w}} : =\sum_i v_i \hat{w}_i.
\end{align}
\end{definition}

\begin{proposition} \label{uwProps}
The functions $\hat{u}_i: \mathcal{S} \rightarrow \Reals$  and $\hat{w}_i: \mathcal{S} \rightarrow \Reals$ have the following properties
\begin{enumerate}
\item \label{SmoothInTheta}They depend smoothly on $\theta$. 

\item \label{CompactSupportW} The functions $\hat{w}_i$  are compactly supported on $\mathcal{W}_{\leq 2}$. 
\item \label{uwProjections} It holds that:

\begin{enumerate}
\item \label{uwPa} $ \int_{\mathcal{S}} \hat{w}_1 \phi_y =  - 2 \pi \sin (\theta) =  \int_{\mathcal{S}} \hat{w}_2 \phi_y$.\\ 
\item \label{uwPb} $\int_{\mathcal{S}} \hat{w}_2  = -  2 \pi = -   \int_{\mathcal{S}} \hat{w}_2 $.
\end{enumerate}
\item \label{uwAsVals} It holds that 
\begin{align} \notag
\|\hat{u}_1 \circ \mathcal{W}_1 - \cos (\theta): C^{j, \alpha} (\Omega^+_{\geq 2}, e^{-s}) \|  & \leq C \theta. \\ \notag
\|\hat{u}_2 \circ \mathcal{W}_2 - \cos (\theta): C^{j, \alpha} (\Omega^+_{\geq 2}, e^{-s}) \|  & \leq C \theta.
\end{align}
\end{enumerate}
\end{proposition}
\begin{proof}
(\ref{SmoothInTheta}) and (\ref{CompactSupportW}) follow directly by from the definition of the maps $\mathcal{Z}[\underline{\varphi}]$ and smooth dependence. To prove (\ref{uwProjections}), note that integration by parts gives
\begin{align} \label{uwP1}
\int_{\mathcal{S}} \hat{w}_i \phi_y =\int_{\mathcal{S}}  \phi_y \mathcal{L} \hat{u}_i =     \lim_{N \rightarrow \infty} \int_{ \partial \mathcal{S}_{\leq N}} \phi_y \partial_s \hat{u}_i.
\end{align}
 Recall that $\hat{u}_i$, $i = 1, 2$ is supported on $\mathcal{W}_i$ and on $\mathcal{W}_{i \geq 1 }$. We have
\begin{align} \notag
\hat{u}_i &= \partial_\sigma \left( H^+[\beta_{\mathcal{S}, i} + \sigma] + (-1)^i f_\mathcal{W} e_z [\beta_{\mathcal{S}, i} + \sigma]\right) \cdot \nu[\mathcal{S}]\\ \notag
& = s \langle e_z[\beta_{\mathcal{S}, i}] , \nu[\mathcal{S}] \rangle.
\end{align}
From (\ref{uwP1}) it then follows
\begin{align} \notag
\int_{\mathcal{S}} \hat{w}_i \phi_y   =  - 2\pi \sin(\theta).
\end{align}
This gives the claims in (\ref{uwPa}), and those in (\ref{uwPb}) follow similarly. The claims in (\ref{uwAsVals}) follow  from Proposition \ref{WingGeometry}.
\end{proof}

\begin{proposition} \label{MeanCurvatureStructure}
There  are constants $C > 0$,  $\delta_0 > 0 $,  $\bar{\theta} > 0 $ so that: Given $\epsilon  > 0$ there is $\bar{\tau} = \bar{\tau} (\epsilon) > 0 $  and $C_0 :  = C_0(\epsilon)$ so that, for $\tau \in [0, \bar{\tau})$, $ \theta \in [0, \bar{\theta})$, and $ \| u : \mathcal{X}^2 \|,  |\underline{\varphi}|  \in [0, \delta_0)$: We can write
\begin{align} \notag
\varrho^* H^*[\underline{\varphi}, u]  = \mathcal{L}\left( \tau \xi^\perp + \underline{\varphi} \cdot \hat{u} + u\right)  + R^*[\underline{\varphi}, u]
\end{align}
where $R^*[\underline{\varphi}, u]$ satisfies the estimate:
\begin{align} \notag
\| R^*[\underline{\varphi}, u]: \mathcal{X}^0\| \leq \epsilon \tau \theta +  C \tau \| u : \mathcal{X}^2\|  + C_0\left(    \| u : \mathcal{X}^2 \|^2+ |\underline{\varphi}|^2) \right)
\end{align}
\end{proposition}

\begin{proposition}\label{MCSmoothDependence}
Given a compact set $K \subset \Reals^3$, there is a constant $C_1 = C_1(K)$ such that:  $H^*[\underline{\varphi}, u]$ is a smooth function of $\tau$, $\theta$, $\underline{\varphi}$ and $(x, s)$  supported on $\mathcal{W}$ and with $C^{j, \alpha}$ norm bounded on $\mathcal{W} \cap K$ by $C_1$. 
\end{proposition}

\begin{proof}[Proof of Proposition \ref{MeanCurvatureStructure}]
We can write
\begin{align} 
 H[\mathcal{K}_i +\varrho_i u \nu[\mathcal{K}_i]] & =  H[\mathcal{K}_i] + \mathcal{L}[\mathcal{K}_i] \varrho_i u + R^{(1)}_{H, \mathcal{K}_i} (\varrho_i u) \\
 & =   H[\mathcal{K}_i] +\varrho_i^{-2} \mathcal{L} \varrho_i u + R^{(1)}_{H, \mathcal{K}_i} (\varrho_i u)  + O (\tau \varrho_i^{-1} \|u\|_{2, \alpha}),
\end{align}
where the last equality above follows from Proposition \ref{BentWingGeometry} (\ref{BentWingSO}). Using Proposition \ref{HQEstimates} and \ref{BentWingGeometry}, we have that
\begin{align} \notag
\|R^{(1)}_{H, \mathcal{K}_i} (\varrho_i u) : C^{j, \alpha}(\Omega^+_{\geq 1}, \varrho_i^{-1}) \| \leq C \| u : \mathcal{X}^2\|^2. 
\end{align}
Moreover we have from Proposition \ref{BentWingGeometry} (\ref{BentWingMC}) that: Given $\delta > 0$ there is $N > 0$ so that 
\begin{align}
\| H^*[\mathcal{K}_i] : C^{j, \alpha}  (\Omega^+_{\geq N}, \cosh^{3/4}(s))\| \leq \delta \tau \theta,
\end{align}
 Combining gives that on $\mathcal{W}_i$ we have
\begin{align} \notag
\varrho^* H^*[\underline{\varphi}, u] = \mathcal{L}u + O(\tau \| u : \mathcal{X}^2\|)  + O( \| u : \mathcal{X}^2\|^2).
\end{align}
From this, it immediately follows that the estimate holds on $\mathcal{W}_{i \geq N}$. Now, given $N$, it follows from Propositions \ref{MCVariationField} and \ref{MCSmoothDependence}  that 
\begin{align} \notag
\varrho^*H^*[\underline{\varphi}, u] = \mathcal{L}(\tau \xi^\perp + \underline{\varphi}\cdot \hat{u} + u)  + R,
\end{align}
where 
\begin{align}
\| R: C^{0, \alpha}(\mathcal{S}_{\leq N}, 1) \| \leq  C_1(N) (\tau^2 \theta + \| u\|^2_{2, \alpha} + |\underline{\varphi}|^2). 
\end{align}
Choosing $\tau$ so that $C_1(N) \tau \leq \epsilon$ then gives the claim. 
\end{proof}

\begin{proposition} \label{KernelProjection}
There is a constant $\epsilon_1 > 0$ so that
\begin{align} \notag
\int_{\mathcal{S}}\mathcal{L} \xi^\perp \phi_y > \epsilon_1 \theta^2. 
\end{align}
\end{proposition}

\begin{proof}
In following we set
\begin{align}
\mathcal{S}_0 = \Sigma \cap \{ |x| \leq \pi, z \geq 0 \}. 
\end{align}
Observe that
\begin{align}
\int_{\mathcal{S}}\mathcal{L} \xi^\perp \phi_y = \int_{\mathcal{S}_0}\mathcal{L} \xi^\perp \phi_y.
\end{align}
We begin by computing the variation field $\xi^\perp$ explicitly. We have from Definition \ref{BendingMap} that
\begin{align}
\left. \partial_\tau  B \right|_{\tau = 0} (x, y, z) = yx e_x - 1/2 x^2 e_y.
\end{align}
Similarly, it follows from Definition \ref{RenormKappaDefs} that
\begin{align} 
\left. \partial_\tau \tilde{\kappa}[\beta] \right|_{\tau = 0} (x, s) &= 1/2\left.  \nabla^2 \kappa[\beta] \right|_{(0, 0)} [(x,s), (x, s)] \\ \notag
&  =  1/2(x^2 - s^2) e_y +  \cos(\beta) x s e_x.
\end{align}
Write 
\begin{align}  \label{KP1}
I_N = \int_{\partial \mathcal{S}_{0  \leq N}} \xi^\perp_{\eta} \phi_y - \xi^\perp \phi_{y, \eta},
\end{align}
were here the subscript `` $\eta$ '' denotes the partial derivative with respect to the outward pointing co-normal at the boundary. We decompose the boundary of $\mathcal{S}_0$ into the following sets
\begin{align}
A_{\pm} : = \mathcal{S}_{0  \leq N} \cap \{ x = \pm \pi \}, \quad B_{i} : = \mathcal{W}_{i}(\{s= N\}) \quad C : = \mathcal{S}_0 \cap \{ z = 0\}.
\end{align}
(Recall the definition of the maps $\mathcal{W}_i (x, s)$ in Proposition \ref{WingGeometry}). We then have that $\partial \mathcal{S}_{ 0  \leq N} =  A_\pm \cup B_i \cup C$.  Note that  the symmetries of the surface $\mathcal{S}_0$ and the perturbation field give that the part of the integral on the right hand side of (\ref{KP1}) vanishes on $C$:
\begin{align} \notag
\int_{C} \xi^\perp_{\eta} \phi_y - \xi^\perp \phi_{y, \eta}= 0.
\end{align}
 From this, it then follows that $I_N$ is a uniformly smooth function of $\theta$ and extends smoothly to $\theta = 0$. Additionally, since $\nu$ converges smoothly to $e_z$ on $A$ and $B$, and since $\partial_\tau B$ is orthogonal to $e_z$, $I_N$ vanishes to first order in $\theta$, and we have
\begin{align} \notag
\ddot{I}_N : = \int_{\partial \mathcal{S}_{0 \leq N}} \dot{\xi^\perp}_{\eta} \dot{\phi}_y - \dot{\xi^\perp} \dot{\phi}_{y, \eta}
\end{align}
where above we have used `` $ \dot{}$ '' to indicate derivatives in $\theta$ at $\theta = 0$. Along $A_+$, the outward pointing conormal is $e_x$, and we have  
\begin{align}  \notag
\partial_x \dot{\xi} =  (\partial_x \xi^\perp) \cdot \dot{\nu} + \xi^\perp \cdot( \partial_x \dot{\nu}).
\end{align}
Also, since $\nu_x$ is orthogonal to $e_y$ along $A$, we have
\begin{align} \notag
\partial_x\dot{\xi}^\perp =  - x \phi_y + xy (\partial_x \phi_{x}).
\end{align}
Note that at $\theta = 0$ we have
\begin{align}
\dot{\phi}_x = \partial_x\dot{f}_{\mathcal{S}}, \quad \dot{\phi}_y = \partial_y\dot{f}_{\mathcal{S}},
\end{align} 
where $\dot{f}_{\mathcal{S}}$ is given in Proposition \ref{SmoothConvergenceToPlane}. Since $\dot{f}_{\mathcal{S}}$ is harmonic, it then follows that:
\begin{align}
\partial_x\dot{\phi}_{x } + \partial_y \dot{\phi}_{y } = 0.
\end{align}
 In the following we let $I_{N}( \gamma)$ denote the restriction of the boundary integral $I_{N}$ to a subset $\gamma$. We then have
\begin{align}
\ddot{I}_{N}( A_+) & = - \pi\int_{-N}^N \dot{\phi}_y^2   + \pi \int_{- N}^N y  \partial_x\dot{\phi}_{x} \dot{\phi}_y dy\\ \notag
&   = - \pi\int_{- N}^N \dot{\phi}_y^2 dy   - \pi \int_{- N}^N y \partial_y\dot{\phi}_{y} \dot{\phi}_y dy \\ \notag
& = - \pi\int_{- N}^N \dot{\phi}_y^2 dy   - \pi/2 \int_{- N}^N y \partial_y (\dot{\phi}_{y})^2 dy \\ \notag
&  = - \pi/2\int_{- N}^N \dot{\phi}_y^2 dy - N \pi  (\dot{\phi}_{y}(\pi, N))^2  \\ \notag
\end{align}
Over $B_1$ the outward pointing conormal agrees with $\partial_s$  and we have
\begin{align} \notag
\ddot{I}_{N}( B_1) & = \int_{B_1} (\partial_s \dot{\xi^\perp}) \dot{\phi}_y +   \dot{\xi^\perp}  ( \partial_s \dot{\phi}_{y })  \\ \notag
& = 2 \pi N \dot{\phi}^2_y(\pi, N) + O (N \cosh^{-1} (N)).
\end{align}
We can write
\[
\dot{\phi}_y|_{x=\pi}=\partial_y \big[\log(\cosh y-\cos x)\big]|_{x=\pi}=\frac{\sinh y}{\cosh y-\cos x}|_{x=\pi}=\frac{\sinh y}{1+\cosh y}
\]

Furthermore,
\[
\frac{d}{dy}\left(\dot{\phi}_y|_{x=\pi}\right)=\frac{1}{\cosh y}>0,
\]
so that on the interval $[0,\infty)$, the function $\dot{\phi}_y|_{x=\pi}$ is increasing in $y$. Thus we get
\begin{align} \notag
\ddot{I}_{N}( A_+) + \ddot{I}_{N}( B_1) & = \pi N \dot{\phi}^2_y(\pi, N)  - \pi/2\int_{- N}^N \dot{\phi}_y^2 dy + O (N \cosh^{-1} (N))  \\ \notag
& >  \pi/2 N \dot{\phi}^2_y(\pi, N) + O (N \cosh^{-1} (N))
\end{align}
The remaining boundary integrals are computed similarly. Summing then gives $
\lim_{N \rightarrow \infty}\ddot{I}_N > 0 $, so that 
\begin{align} \notag
\int_{\mathcal{S}_0} \mathcal{L} \xi^\perp \phi_y >  \epsilon_1 \theta^2. 
\end{align}
This completes the proof.
\end{proof}

\section{The linear problem on  $\mathcal{S}$} \label{FlatScherkLinearProblem}

In this section we record the main invertibility result--Proposition \ref{FlatInverse}--for the stability operator on the unmodified Scherk surface. The result characterizes when the problem
\begin{align} \label{FlatScherkProblem}
\mathcal{L} v = E
\end{align}
admits a solution in certain function spaces with decay along the ends of the Scherk surface.  In an ensuing section we show that the linear problem on the initial surface can then be treated as a perturbation of the problem on  $\mathcal{S}$.  Before we state Theorem \ref{FlatInverse} we record a few definitions.

\begin{proposition} \label{FlatInverse}
There is linear map
\begin{align} \notag
\mathcal{R}[ -] : \mathcal{X}^0 \rightarrow \Reals^4 \times \mathcal{X}^2
\end{align}
such that: Given $E \in \mathcal{X}^0$ and with $(\underline{\varphi}, u) = \mathcal{R}[E]$ the following statements hold:

\begin{enumerate}
\item It holds that 
\begin{align} \notag
\mathcal{L}u + \underline{b} \cdot \underline{\hat{w}} = E. 
\end{align}

\item \label{FIBounds} There is a constant $C$ so that 
\begin{align} \notag
\| u : \mathcal{X}^2 \|, |\underline{\varphi}| \leq C \left( \| E : \mathcal{X}^0\|  +\theta^{-1} \langle E, \phi_y \rangle_{\mathcal{S}}\right).
\end{align}

\item \label{varphiRelations} It holds that 
\begin{align} \notag
\varphi_1+  \varphi_2 =  \frac{ - 1}{2 \pi\sin(\theta)} \int_{\mathcal{S}} E \phi_y, \quad \varphi_1 -  \varphi_2 =   \frac{ - 1}{2 \pi} \int_{\mathcal{S}} E. 
\end{align}
\end{enumerate}
\end{proposition}

We record the proof of Proposition \ref{FlatInverse} in three main stages. In \ref{Orthogonalizing}, we show that linear combinations of the  functions $\hat{w}_1$ and $\hat{w}_2$ can be added to the error term $E$ to achieve $L^2$ orthogonality to $1$ and $\phi_y$. Due to the non-uniform (in $\theta$) projection of $\hat{w}_i$ onto $\phi_y$, the resulting orthogonalized error term has size that is no longer commensurate with that of $E$, and is the reason for the right hand side of the estimate in  Proposition \ref{FlatInverse}(\ref{FIBounds}).  In Section \ref{ReducingSupport}, we show that the proof can be reduced to the case of considering inhomogeneous terms with support on $\mathcal{C}_0$. This is essentially a straightforward consequence of the almost flat geometry of $\mathcal{W}$ and the invertibility result for the flat laplacian on cylinders recorded in Proposition \ref{LaplaceInverseModDomain}. In Section \ref{ReducedSupportProof}, we record the proof for $\mathcal{C}_0$ supported inhomogeneous terms. 

\subsection{Orthogonalizing the error term} \label{Orthogonalizing}
 Set $f : = \hat{w}_1 +  \hat{w}_1 $ and $g: = \hat{w}_1 - \hat{w}_2$. It follows directly from Proposition \ref{uwProps} that
 \begin{align}  \notag
 \int_{\mathcal{S}} f \phi_y = - 4 \pi \sin (\theta), \quad \int_{\mathcal{S}} f = 0, \quad  \int_{\mathcal{S}} g= - 4 \pi, \quad \int_{\mathcal{S} } g\phi_y = 0.
 \end{align}
 Thus, with 
 \begin{align} \notag
 a : = \left(\frac{1}{4\pi \sin (\theta)} \ \int_{\mathcal{S}}  E\phi_y \right) , \quad b : = \frac{1}{4\pi } \ \int_{\mathcal{S}}  E,
 \end{align}
  the function $E_0 : = E + a f + b g  $ is $L^2$ orthogonal to $1$ and $\phi_y$ on $\mathcal{S}$. Since $\hat{w}_i$ are smooth functions uniformly bounded in $\theta$, we have that 
\begin{align} \notag
\| E_0 : \mathcal{X}^0\|  \leq C \left(\| E : \mathcal{X}^0\| +  \frac{1}{\sin(\theta)}\int_{\mathcal{S}}  E\phi_y  \right).
\end{align}
We conclude by observing that we can write $E_0 =  \varphi_1 \hat{w}_1 + \varphi_2 \hat{w}_2$, were $\varphi_1 = a + b$, $\varphi_2 : = a- b$. It then follows easily that $\varphi_1$ and $\varphi_2$ satisfy the relations in (\ref{varphiRelations}).

\subsection{Reducing to $\mathcal{C}$ supported inhomogeneous error terms} \label{ReducingSupport}
The reader my wish recall the definitions of $\mathcal{C}_0$ and $\mathcal{C}$ in Proposition \ref{ScherkNearAxis}. We first observe that:
\begin{proposition} \label{OperatorLaplaceComparison} \notag
It holds that
\begin{align}
\| \mathcal{L} - \Delta_{\Omega}: C^{k, \alpha} \left(\mathcal{W}, \frac{1}{\cosh(s)}\right)\| \leq C \sin (\theta).
\end{align}
\end{proposition}
Let $D^c[\mathcal{W}]$ denote the complement of $D[\mathcal{W}]$ in $\Omega$ and set
\begin{align} \notag
v_0: = \mathcal{R}_3[\mathcal{L}, D^c[\mathcal{W}], E_0].
\end{align}
 (Recall Proposition \ref{LaplacePerturbationInverseModDomain}).  The function $v_0$ is then  well-defined and satisfies the weighted estimate:
 \begin{align}
 \|v_0 : C^{2, \alpha} (\mathcal{W}, \cosh^{-\gamma }(s))\| & \leq C\|  E_0 :   C^{j, \alpha} (\mathcal{W}, \cosh^{-\gamma }(s)) \| \\ \notag
 & \leq C \| E_0 :  \mathcal{X}^0 \|.
 \end{align}
Let $\psi: D[\mathcal{W}] \rightarrow \Reals$ be the cutoff function determined as follows:
\begin{align}
\psi (x, s) : = \psi_0[\epsilon_0/2, \epsilon_0] (r) 
\end{align}
 where above we have set $r : = \sqrt{x^2 + r^2}$ and $\epsilon_0$ is as in Definition \ref{WingMapDefinition}. We then set
 \begin{align}
 \psi^*  (p): = \psi \circ \mathcal{W}^{-1} (p), \quad p \in \mathcal{W}.
 \end{align}
  Then it is directly verified that the gradient of $\psi^*$ is supported on $\mathcal{C} \cap \mathcal{W}$, and we have
 \begin{align} \notag
 \| \psi^*:  C^{2, \alpha}( \mathcal{W} \cap \mathcal{C})\| \leq C,
 \end{align}
 where $C$ is a universal constant independent of $\theta$. We then set 
\begin{align} \notag
E_1 : = E_0 - \mathcal{L} (\psi^* v_0).
\end{align}
We then have

\begin{enumerate}
\item $\int_{\mathcal{S}} E_1   =  \int_{\mathcal{S}} E - \mathcal{L} (\psi^* v_0) = 0$. \\

\item$ \int_{\mathcal{S}} E_1 \phi_y   =  \int_{\mathcal{S}} E \phi_y - \phi_y\mathcal{L} (\psi^* v_0) = 0.$
\end{enumerate}

Thus, the inhomogeneous term $E_1$ satisfies the same orthogonality conditions as $E_0$ is supported on $\mathcal{C}_0$. 

\subsection{Solving for inhomogeneous terms supported on $\mathcal{C}_0$} \label{ReducedSupportProof}

We can now without loss of generality assume that our inhomogeneous term $E$ in (\ref{FlatScherkProblem}) is supported on $\mathcal{C}_0$. This allows us to conformally move the linear problem we wish to solve to a simpler object, namely the catenoid $\mathcal{C}_0$ without changing the error term appreciably. To do this, we make the following definition:

\begin{definition} \label{ScherkToCatenoidMap}
We let $M: D[\mathcal{C}] \rightarrow \Omega$ be the map given by
\begin{align} \notag
M(x, s) : = \nu^{-1}[\mathcal{C}_0] \circ\nu[\mathcal{S}] (x, s). 
\end{align}
\end{definition}

\begin{proposition} \label{ScherkToCatenoidMapProps}
The map $M$ given Definition \ref{ScherkToCatenoidMap} has the following properties
\begin{enumerate}
\item It is a conformal diffeomorphism onto its image.

\item Its conformal factor $\varrho[M]$ is given by
\begin{align}\notag
\varrho[M] : = |A[\mathcal{S}]|/|A[\mathcal{C}_0]|.
\end{align}

\item \label{STCMP3} It holds that 
\begin{align} \notag
\| \varrho[M] - 1: C^{0, \alpha}(D[\mathcal{C}], \cosh^2(s))\| \leq  C \sin^2 (\theta).
\end{align}

\item \label{STCMP4}It holds that 
\begin{align} \notag
\| M - \Id : C^{0, \alpha}(D[\mathcal{C}], \cosh(s))\| \leq  C \sin^2 (\theta).
\end{align}
\end{enumerate}
\end{proposition}

\begin{proof}
Let $\nu_0$ denote the unit normal on the catenoid $\mathcal{C}_0$ and $A_0$ the second fundamental form, so that 
\begin{align}
A_0 =  - 2  dx^2 + 2 ds^2.
\end{align}
Let $\Phi: \mathbb{R}^2\times \Omega \rightarrow \Reals^2$  be the map given by
\begin{align} \notag
\Phi (m, x, s) : = \left(\nu[\mathcal{C}] (x, s) - \nu_0  \circ (m + (x, s))  \right)^\parallel
\end{align}
where $m \in \Reals^2$, and where the superscript `` $\parallel$ '' denotes the projection onto the tangent plane of $\mathcal{C}_0$. From Corollary \ref{CatQuantities} we have 
\begin{align} \notag
  \partial_m \Phi (0, x, s) = \left.\nabla\nu_0 \right|_m (x, s) = \left. A_0 \right|_M,  \quad \|\Phi (0, x, s)\|_{j, \alpha} \leq C \theta^2 \cosh(s) \leq \epsilon_0 \theta,
\end{align}
where $\epsilon_0$ is as in Proposition \ref{ScherkNearAxis}.  Proposition \ref{IFTHandyVersion} then gives a function $(x, s) \mapsto m(x, s)$ so that 
with $M(x, s) : = (x ,s) + m(x, s)$ we have
\begin{align}
\nu[\mathcal{C}] (x, s) - \nu_0  \circ (M(x, s)) = 0,
\end{align}
which gives Claim  (\ref{STCMP4}). Claim (\ref{STCMP3}) follows by writing
\begin{align}
\varrho[M] -1  = |A[\mathcal{C}_0]|^{-1} \left( |A[\mathcal{S}]|^{-1} - |A[\mathcal{C}_0]|^{-1}\right)
\end{align}
and using Corollary \ref{CatQuantities}.
\end{proof}
Now, instead of solving (\ref{FlatScherkProblem}) directly, we first lift the problem to the sphere using the Gauss map of $\mathcal{S}$, which gives the equivalent form
\begin{align}
\left(\Delta[S^2]  + 2 \right) v = 2 E_1/|A[\mathcal{S}]|. 
\end{align}
Applying the inverse Gauss map of $\mathcal{C}_0$ then gives 
\begin{align} 
\mathcal{L}[\mathcal{C}_0] v  = \left(  |A[\mathcal{S}]|/|A[\mathcal{C}_0]| \right) E_1 : = \tilde{E}_1. 
\end{align}
Observe that the imposed orthogonality conditions on the right hand side are preserved under the conformal changes: 
\begin{align} \notag
&\int_{\Omega} \tilde{ E}_1 d \mu [\mathcal{C}_0]  = \int_{\Omega} E_1 d \mu [\mathcal{S}] = 0. \\ \notag
& \int_{\Omega} \tilde{ E}_1 \phi_yd \mu [\mathcal{C}_0]  = \int_{\Omega} E_1  \phi_y  d \mu [\mathcal{S}]= 0.
\end{align}
Additionally, we have from Lemma \ref{ScherkToCatenoidMapProps} and the support of $E$ that 
\begin{align} \notag
\| \tilde{E}_1 : \mathcal{X}^0[\mathcal{C}_0]\| \leq C \| E_1 : \mathcal{X}^0 \|. 
\end{align}
We can then apply Proposition \ref{CatenoidInverse} to obtain the function $v_1 : = \mathcal{R}[\mathcal{C}_0, \tilde{E}_1]$. We will now abuse notation by identifying $v$ with its pushforward to the sphere and $\mathcal{S}$ under $\nu[\mathcal{C}_0]$ and $M^{-1}$, respectively.  Again, from Lemma \ref{ScherkToCatenoidMapProps}, it holds that 
\begin{align} \notag
\|v_1 : C^{2, \alpha} (D[\mathcal{C}], \cosh^{\gamma}(s)) \| \leq C \| E_1 : \mathcal{X}^0 \|.
\end{align}
Moreover, by Proposition \ref{CatenoidInverse}, we have that the supremum of $v$ is bounded on $\Omega$ by $ C \| E_1 : \mathcal{X}^0 \|$. Thus, standard removable singularity theory gives that $v$ extends to a  smooth function on $S^2$.  Setting $v : = v_1 + \psi v_0$, $\underline{\varphi} : = (\varphi_1, \varphi_2, \varphi_3, \varphi_4)$ and  $\mathcal{R}[ E]  = (v, \underline{\varphi})$ then gives Proposition  \ref{FlatInverse}.

\section{Finding  minimal normal graphs} \label{FindingMinimalGraph}

\begin{proposition}
Set 
\begin{align} \notag
( \underline{\varphi}_0, u_0) : = \mathcal{R}(\mathcal{L} \xi^\perp), 
\end{align}
where $\mathcal{R}$ is as in Proposition \ref{FlatInverse} and $\xi^\perp$ is as in Proposition \ref{MCVariationField}. Then we have that 
\begin{align} \notag
 \| u_0 : \mathcal{X}^2\| \leq C  \theta, \quad |\underline{\varphi}_0| \leq C   \theta.
\end{align}

\end{proposition}

\begin{proof}
From Proposition \ref{HQEstimates}
\end{proof}

\begin{definition} \label{FixedPointSet}
For $\zeta > 0 $ to be determined, set
\begin{align} \notag
\Xi \subset \Reals^4 \times \mathcal{X}^2  : = \{ (\underline{\varphi}, u): \| u : \mathcal{X}^2\| \leq \zeta \tau \theta, \quad |\underline{\varphi}| \leq \zeta \tau \theta\}. 
\end{align}
\end{definition}

\begin{definition} \label{FixedPointMap}
We let $\Psi: \Xi \rightarrow \mathcal{X}^2 \times \Reals^4$  be the function given as follows:
\begin{align} \notag
\Psi(\underline{\varphi}, u) & : = (\underline{\varphi}, u) - \mathcal{R} \varrho^* H^*[\underline{\varphi}, u] \\ \notag
& =  - \tau( \underline{\varphi}_0, u_0) - \mathcal{R} R^*[\underline{\varphi}, u]. 
\end{align}
(Recall the definition of $R^*[\underline{\varphi}, u]$ in Definition \ref{MeanCurvatureStructure}).
\end{definition}

\begin{proposition} \label{YayFixedPoint}
There is  $\zeta > 0$ sufficiently large and $\bar{\tau}  > 0$ so that, for  $\tau \in [0, \bar{\tau})$, the following statements hold
\begin{enumerate}
\item $\Psi$ has a fixed point $( \underline{\varphi}^*, u^*)$ in $\Xi$. 
\end{enumerate}
\end{proposition}

\begin{proposition} \label{FixedPointRemainder}
Given $\zeta > 0$ in Definition \ref{FixedPointSet}, $\epsilon > 0$ and $\gamma$ in Definition \ref{ScherkFunctionSpaces} belonging to the interval $(1/2, 1)$, there are $\bar{\tau} > 0$ and $\bar{\theta} > 0$ so that: for $\tau \in [0, \bar{\tau})$, $\theta \in (0, \bar{\theta})$ and $(\underline{\varphi}, u) \in \Xi$, the following estimates hold:

\begin{enumerate}
\item \label{RemainderTotalSize}$\| R^*[\underline{\varphi}, u] : \mathcal{X}^0\| \leq \epsilon \tau \theta$.\\
\item \label{RemainderKernelSize} $\left|\int_{\mathcal{S}} R^*[\underline{\varphi}, u] \phi_y \right| \leq \epsilon \tau \theta^2$. 
\end{enumerate}
\end{proposition} 

\begin{proof}
Estimate (\ref{RemainderTotalSize}) follows from Theorem \ref{MeanCurvatureStructure} by taking $\zeta \tau C_0(\epsilon) \leq \delta$.  To prove Estimate (\ref{RemainderKernelSize}),  We write
\begin{align} 
\left|\int_{\mathcal{S}_0} R^*[\underline{\varphi}, u] \phi_y \right| & \leq \left|\int_{\mathcal{W}} R^*[\underline{\varphi}, u] \phi_y \right| + \left|\int_{\mathcal{C}} R^*[\underline{\varphi}, u] \phi_y \right| \\ \notag
&= I + II.
\end{align} 
Using that $\phi_y$ is bounded by  a constant times $\theta$ on $\mathcal{W}$, we have that 
\begin{align}
I \leq C \epsilon \tau \theta^2.
\end{align}
To estimate $II$, note that on $\mathcal{C}$,  $R^*[\underline{\varphi}, u]$ is independent of $\tau$, $\theta$ and $\underline{\varphi}$ and we have that $
R^*[\underline{\varphi}, u] : = R^{1}_{H, \mathcal{C}}(u)$. Propositions \ref{HQEstimates} and  \ref{ScherkNearAxisGraph} then give that 
\begin{align}
\| R^{1}_{H, \mathcal{C}}(u)\|_{0, \alpha} & \leq C \theta^{-3} \cosh^{-3}(s) \| u \|_{2, \alpha} \\ \notag
&\leq C\| u: \mathcal{X}^2\|^2 \theta^{2 \gamma - 3} \cosh^{2 \gamma - 3} (s).
\end{align}
We then have
\begin{align} \notag
II  & = \int_{\mathcal{C}} \left(R^{1}_{H, \mathcal{C}}(u) \right) \left( \phi_y \right) d\mu[\mathcal{S}] \\ \notag
& \leq C \| u: \mathcal{X}^2\|^2 \int_{D[\mathcal{C}]} \left(\theta^{2 \gamma - 3} \cosh^{2 \gamma - 3} (s) \right) \left( \cosh^{-1}(s)\right)\left( \theta^2 \cosh^{2}(s)\right) d\mu[\Omega] \\ \notag
&  \leq  C \| u: \mathcal{X}^2\|^2 \theta^{2 \gamma - 1} \int_0^{\mathrm{arcosh}(\delta_0/\theta)} \cosh^{2 \gamma - 2} (s) ds \\ \notag
& \leq C \zeta^2 \tau^2 \theta^{2 \gamma + 1} \\ \notag
& \leq C \zeta^2 \tau^2 \theta^{2 }
\end{align} 
(when $\gamma \geq 1/2$). Taking $C \zeta^2 \tau \leq \epsilon$ then gives the claim. 
\end{proof}

\begin{proof} [Proof of Proposition \ref{YayFixedPoint}]
We have that 
\begin{align}\notag
\Phi (\underline{\varphi}, u) + \tau (\underline{\varphi}_0, u_0) = - \mathcal{R} R^*[\underline{\varphi}, u] 
\end{align}
Choose $\zeta$ so that 
\begin{align}
\|(\underline{\varphi}_0, u_0) \| \leq \zeta  \theta/2.
\end{align}
Given $\epsilon'$ we can then choose $\epsilon$ in Proposition \ref{FixedPointRemainder} so that
\begin{align} \notag
\| \mathcal{R} R^*[\underline{\varphi}, u]: \Reals^4 \times \mathcal{X}^2 \ \| \leq \epsilon' \tau \theta.
\end{align}

  It then follows that $\Psi(\Xi) \subset (\Xi)$. The Schauder fixed point theorem (Proposition \ref{SchauderFixedPointTheorem}) then gives that $\Psi$ has at least one fixed point on $\Xi$, which we denote by $(\underline{\varphi}^*, u^*)$. We then have that $S^*[\underline{\varphi}^*, u^*]$ is a complete immersed minimal surface. Convergence, completeness, properness,  as well as quantitative bounds on the convergence rates of parts of the surface to the singular object, as described qualitatively in the statement of the main theorem (Theorem \ref{MainTheorem}) follow directly from the geometry of the initial surfaces and the bounds built into the function spaces in the preceding sections.
\end{proof}

\clearpage
\begin{table}[ht]
\caption{Basic Notational Conventions}
\begin{tabular}{l l l}
\hline\hline
Symbol & Content & Ref.\\ [0.5ex]
\hline
$(x,y,z)$ &coordinates on Euclidean 3-space $\Reals^3$ \\
$(x,s)$ &coordinates on $\Reals^2$, or on $\Omega_0$ ($x$ is $2\pi$-periodic)&\\
$Q_i$& the four $xy$-quadrants of $\Reals^3$ \\
$e_x, e_y, e_z$ & the standard unit vectors of $\Reals^3$\\
$e_{(t)}$ & point at angle $t$ on the unit circle in the $xy$-plane\\
$e_y[\beta], e_z[\beta]$ & rotated unit vectors\\
$\mathfrak{R}_x$, $\mathfrak{R}_y$,  $\mathfrak{R}_z$& reflections through coordinate planes\\
$\mathfrak{T}_t$ &translation by $ t e_x$\\
$\mathfrak{T}^*_t$ &related rotation\\
$\mathfrak{G}$ & group generated by $\mathfrak{R}_x$, $\mathfrak{R}_z$ and $\mathfrak{T}_{2 \pi}$\\
$\mathfrak{G}^*$& group generated by $\mathfrak{R}_x$, $\mathfrak{R}_z$ and $\mathfrak{T}^*_{2 \pi}$\\
$\mathbb{E}$& quotient of $\Reals^3$ by $\mathfrak{G}$\\
$\mathbb{E}^*$& quotient of $\Reals^3$ by $\mathfrak{G}^*$\\
$\Omega_0$& flat two-dimensional cylinder\\
$\Omega_0^\pm$& the $\{\pm s\geq 0\}$ part of the flat cylinder\\
$H^\pm$ & the $\{\pm s\geq 0\}$ half-spaces in $\Reals^2$ \\
$s$ & parameter along (a surface parametrized by $\Reals^2$, such as) $\Omega_0$\\
$U_{\leq c}$ & indication of $s$-sublevel set, i.e. $U \cap \{ s \leq c \}$\\
$\psi[a,b]$& smooth cut-off function in one variable\\
$g[S]$ & metric on the surface $S$\\
$\Gamma^k_{ij}$ & Christoffel symbols\\
$A[S]$ & second fundamental form of the surface $S$\\
$\nu[S]$ & unit normal vector to the surface $S$\\
[1ex]
\hline
\end{tabular}
\end{table}

\begin{table}[ht]
\caption{Notation Specific to the Construction}
\begin{tabular}{l l l}
\hline\hline
Symbol & Content & Ref.\\ [0.5ex]
\hline
$\tau$ & small parameter in the construction (genus $\simeq\tau^{-1}$)\\
$\theta$ & small angle parameter in the construction\\
$\underline{\varphi}$ & parameter vector with components $(d_1,d_2,b_1,b_2)$\\
$\epsilon_0$ & small cut-off (related to Scherk geometry)\\
$C$ & large positive constants\\
$\delta$ & small positive constant\\
$\mathcal{L}_S$& minimal surface stability operator $\Delta_S + |A_S|^2$\\
$E$& source term in linearized equation\\
$\mathcal{C}_0$ & the catenoid of neck width 1\\
$\mathcal{C}$ & the catenoid of neck width 1\\
$\phi_{\mathcal{C}_0}$ & conformal parametrization of the catenoid\\
$\kappa$ & conformal parametrization of the catenoid (or plane)\\
$\tilde{\kappa}$ & renormalized versions of the $\kappa$ \\
$\rho[\beta](s)$ & conformal factor of the catenoid (or plane, for $\beta=0$)\\
$\tilde{\rho}$ & conformal factor of the renormalized maps $\tilde{\kappa}.$\\
$\{e^\kappa\}$ & orthonormal frame $\{e_i(x,s)\}$ on the catenoid\\
$\{\tilde{e}^\kappa\}$ & transformed orthonormal frame $\{e_i(\tau x,\tau s)\}$\\
$T_x, T_s$ & derivative matrices\\
$\Sigma$&Scherk tower\\
$\mathcal{S}$& quotient of $\Sigma$ by $\mathfrak{G}$\\
$\phi_x,\phi_y,\phi_z$ & Killing functions on the Scherk towers (= $e_i\cdot \nu$)\\
$h_\mathcal{S}$ & affine offset for Scherk towers\\
$\dot{f}_\mathcal{S}$ & harmonic function approx. Scherk towers ($\theta\simeq 0$)\\
$f_\mathcal{W}$ & function realized the Scherk wings as graphs over affine planes \\
$\mathcal{W}_i$ & $i$'th wing of Scherk tower\\
$\mathcal{W}$ & the ``planar part'' of the Scherk surface $\Sigma$.   \\
$\mathcal{\beta}_{\mathcal{S},i}$ & angles directing the $i^{th}$ wing $\mathcal{W}_i$ \\
$B$ & bending maps\\
$\tilde{B}$ & modified bending maps\\
$\mathcal{X}^0$, $\mathcal{X}^2$ & weighted H\"o{}lder spaces\\
$D[\mathcal{W}]$& domain\\
$\psi$ & smooth cut-off for the localization of the linear problem\\
[1ex]
\hline
\end{tabular}
\end{table}
\clearpage

\bibliographystyle{amsalpha}

\begin{thebibliography}{99}

\bibitem[AIC95]{AI}
 S. Angenent, T. Ilmanen, D. Chopp,
 \emph{A computed example of nonuniqueness of mean curvature flow in $\mathbb{R}^3$}, Comm. Partial Differential Equations \textbf{20} (1995), no. 11--12, 1937--1958. 

\bibitem[Co84]{Co84}
 C. Costa,
 \emph{Example of a complete embedded minimal Immersion of genus one and three embedded ends.}, Bull Soc. Bras. Mat., {\bf 15}, 47-54 (1984).

\bibitem[Ev97]{Ev}
 L.C. Evans,
 \emph{Partial differential equations}, AMS, 1997.

\bibitem[HM85]{HM85}
 D. Hoffman, W.H. Meeks III,
 \emph{A complete embedded minimal surface in $\mathbb{R}^3$ with genus one and three ends}, 
J. Differential Geom. \textbf{21} (1985), no. 1, 109--127.

\bibitem[HM90a]{HM90a}
 D. Hoffman, W. H. Meeks III,
 \emph{Embedded minimal surfaces of finite topology}, Ann. of Math. \textbf{131} (1990) 1--34.

\bibitem[HM90b]{HM90b}
 D. Hoffman, W.H. Meeks III,
 \emph{The strong half-space theorem for minimal surfaces}, Invent. Math. \textbf{101} (1990), 373--377.

\bibitem[HK]{HK}
 D. Hoffman and H. Karcher,
 \emph{Complete embedded minimal surfaces of finite total curvature},  Encyclopaedia of Mathematical Sciences Volume 90, 1997, pp 5-93

\bibitem[HW]{HW}
 D. Hoffman and B. White,
 \emph{Genus-one helicoids from a variational point of view}, Comment. Math. Helv. {\bf 83} (2008), no. 4, 767-813.

\bibitem[HTW1]{HTW1}
 D. Hoffman, M. Traizet, B. White,
 \emph{Helicoidal Minimal surfaces of prescribed genus I}, Preprint.   http://arxiv.org/pdf/1304.5861v1.pdf.

\bibitem[HTW2]{HTW2}
 D. Hoffman, M. Traizet, B. White,
 \emph{Helicoidal Minimal surfaces of prescribed genus II}, Preprint. http://arxiv.org/pdf/1304.6180v1.pdf.

\bibitem[Jo02]{Jo}
 J. Jost,
 \emph{Partial differential equations}, Springer, 2002.

\bibitem[Ka95]{Ka95}
 N. Kapouleas,
 \emph{Constant mean curvature surfaces by fusing Wente tori}, Invent. Math, \textbf{119} (1995), 443-518.

\bibitem[Ka97]{Ka97}
 N. Kapouleas,
 \emph{Complete embedded minimal surfaces of finite total curvature}, J. Differential Geom. \textbf{47} (1997), no. 1, 95--169.

\bibitem[Ka05]{Ka05}
 N. Kapouleas,
 \emph{Constructions of minimal surfaces by gluing minimal immersions}, Global theory of minimal surfaces, 489--524, Clay Math. Proc., \textbf{2}, Amer. Math. Soc., Providence, RI, 2005.

\bibitem[Ka11]{Ka11}
 N. Kapouleas,
 \emph{Doubling and Desingularization Constructions for Minimal Surfaces}, Volume in honor of Professor Richard M. Schoen's 60th birthday, arXiv:1012.5788v1.
 
\bibitem[KKM1]{KKM1} 
 N. Kapouleas, S. J. Kleene, N . M.  M\o ller,
 \emph{Mean curvature self shrinkers of high genus: Non-compact examples}, to appear in J. Reine Angew. Math (2012). 
 
\bibitem[MR1]{MR1} 
 W. Meeks III and H. Rosenberg,
 \emph{The uniqueness of the helicoid}, Ann. of Math. {\bf 161} (2005) 727-758.
 
\bibitem[Ros06]{Ros06} 
 A. Ros,
 \emph{Complete embedded minimal surfaces with finite topology},
 International Congress of Mathematicians. Vol. II, Eur. Math. Soc., ZŸrich (2006), 907Ð-926.
 
\bibitem[Sc02]{Sc}
 R. Schoen,
 \emph{Uniqueness, symmetry and embeddedness of minimal surfaces}, J. Dff. Geom. {\bf 60} (2002), 103 -53.

 
\bibitem[Tr96]{Tr96}
 M. Traizet,
 \emph{Construction de surfaces minimales en recollant des surfaces de Scherk}, Ann. Inst. Fourier (Grenoble), \textbf{46} (1996), pp. 1385Ð-1442.



\end{thebibliography}

\end{document}